\documentclass[12pt]{amsart}
\usepackage{amsmath}
\usepackage{amsthm}
\usepackage{tikz}
\usepackage{hyperref}
\usepackage[left=2.5cm,right=2.5cm,top=3cm,bottom=3cm]{geometry}
\usepackage{eucal}
\usepackage{palatino}
\usepackage{euler}
\usepackage{mathrsfs}
\usepackage{amssymb}
\usepackage{amscd}
\usepackage{latexsym}
\usepackage{epsfig}
\usepackage{graphicx}
\usepackage{amsfonts}
\usepackage{psfrag}
\usepackage{caption}
\usepackage{rotating}
\usepackage{mathtools}
\usepackage{fullpage}
\usepackage{etoolbox} 

\usepackage[normalem]{ulem}

\usepackage{pifont}
\usepackage{epsfig}

\usepackage{url}
\usepackage{cite}
\usepackage{dsfont}
\usepackage{array}

\newtheorem{theorem}{Theorem}
\newtheorem{definition}[theorem]{Definition}
\newtheorem{proposition}[theorem]{Proposition}
\newtheorem{corollary}[theorem]{Corollary}
\newtheorem{lemma}[theorem]{Lemma}

\theoremstyle{remark}
\newtheorem*{remark}{Remark}

\DeclareMathOperator{\GL}{GL}
\DeclareMathOperator{\SL}{SL}

\DeclareMathOperator{\Adj}{Ad}

\DeclareMathOperator{\adj}{ad}

\DeclareMathOperator{\Hom}{Hom}

\DeclareMathOperator{\Sym}{Sym}

\newtoggle{showColor}
\newtoggle{showNotes}
\toggletrue{showColor}
\togglefalse{showNotes}

\iftoggle{showNotes} {%
	\newcommand{\note}[1]{{\textcolor{red}{$\langle$#1$\rangle$}}} 
}{%
	\newcommand{\note}[1]{}
}

\title{Hessenberg Varieties for the Minimal Nilpotent Orbit}
\author{Hiraku Abe}
\address{Osaka City University Advanced Mathematical Institute, 3-3-138 Sugimoto, Sumiyoshi-ku, Osaka 558-8585, 
JAPAN / Department of Mathematics, University of Toronto, 40 St. George Street, Toronto, Ontario, Canada, M5S 2E4}
\email{~~~abeh@math.mcmaster.ca}

\author{Peter Crooks}
\address{Department of Mathematics, University of Toronto, 40 St. George Street, Toronto, Ontario, Canada, M5S 2E4}
\email{~~~peter.crooks@utoronto.ca}

\begin{document}
\begin{abstract}
For a connected, simply-connected complex simple algebraic group $G$, we examine a class of Hessenberg varieties associated with the minimal nilpotent orbit. In particular, we compute the Poincar\'{e} polynomials and irreducible components of these varieties in Lie type $A$. Furthermore, we show these Hessenberg varieties to be GKM with respect to the action of a maximal torus $T\subseteq G$. The corresponding GKM graphs are then explicitly determined. Finally, we present the ordinary and $T$-equivariant cohomology rings of our varieties as quotients of those of the flag variety.   
\end{abstract}


\maketitle

\section{Introduction}\label{Section: Introduction}
\subsection{Context and Statement of Results}
Hessenberg varieties form a large and interesting family of subvarieties of the flag variety, including Springer fibres, the Peterson variety, the toric variety associated to Weyl chambers, and the flag variety itself. They are studied in the contexts of algebraic geometry \cite{Brosnan, DeMari, Insko, Klyachko, Precup, Procesi, Tymoczko-Decomposing}, combinatorics \cite{Drellich, Harada2, Guay-Paquet,Shareshian}, geometric representation theory \cite{Fung, Springer}, and equivariant algebraic topology. Concerning the last of these areas, there has been a pronounced emphasis on equivariant cohomology computations for torus actions on Hessenberg varieties (see \cite{Abe, Fukukawa, Harada1, Horiguchi}).

This manuscript studies a class of Hessenberg varieties arising from the minimal nilpotent orbit. More precisely, let $G$ be a connected, simply-connected simple algebraic group over $\mathbb{C}$ with Lie algebra $\mathfrak{g}$, opposite Borel subgroups $B,B_{-}\subseteq G$, maximal torus $T=B\cap B_{-}$, Weyl group $W=N_G(T)/T$, and highest root $\theta$. Each highest root vector $e_{\theta}\in\mathfrak{g}_{\theta}\setminus\{0\}$ belongs to the minimal nilpotent orbit of $G$. Accordingly, for a Hessenberg subspace $H\subseteq\mathfrak{g}$, we consider the Hessenberg variety $X_H(e_{\theta})\subseteq G/B$. This variety has received some attention in the literature as an example of a \textit{highest weight Hessenberg variety} (see \cite{Tymoczko-Dimensions}).  

As is the case with nilpotent Hessenberg varieties in general, $X_H(e_{\theta})$ is sometimes singular and reducible, and its geometry depends heavily on the choice of $H$. However, one distinguishing feature is that $X_H(e_{\theta})$ is a union of Schubert varieties. In particular, it is invariant under the action of $T$ on $G/B$. 

While we present a wide array of results on the geometry and topology of $X_H(e_{\theta})$, the following are our main results.

\begin{itemize}
\item There are explicit combinatorial procedures for determining the Poincar\'{e} polynomial and irreducible components of $X_H(e_{\theta})$ in Lie type $A$.  	

\item The $T$-action renders $X_H(e_{\theta})$ a GKM variety. Its GKM graph is the full subgraph of the GKM graph of $G/B$ with vertex set $\{w\in W:\mathfrak{g}_{w^{-1}\theta}\subseteq H\}$.

\item The restriction map $i_T^*:H_T^*(G/B;\mathbb{Q})\rightarrow H_T^*(X_H(e_{\theta});\mathbb{Q})$ is surjective. Its kernel is the $H_T^*(\text{pt};\mathbb{Q})$-submodule of $H_T^*(G/B;\mathbb{Q})$ freely generated by the ``equivariant opposite Schubert classes'' $\sigma_T(w)$ for all $w\in W$ with $\mathfrak{g}_{w^{-1}\theta}\cap H=\{0\}$.\footnote{Here, $\sigma_T(w)\in H_T^{2\ell(w)}(G/B;\mathbb{Q})$ is the class determined by the opposite Schubert variety $\overline{B_{-}wB/B}$.}
\end{itemize}

We also prove a similar statement for the ordinary cohomology ring $H^*(X_H(e_{\theta});\mathbb{Q})$.

\subsection{Structure of the Article}
We begin with \ref{Section: Notation}, which enumerates some of the important objects used throughout the article. Section \ref{Section: Minimal Nilpotent Hessenberg Variety} then properly introduces $X_H(e_{\theta})$. In \ref{Section: Examples in Type A}, we use a common description of Hessenberg varieties in type $A$ to provide an explicit example of $X_H(e_{\theta})$.

Section \ref{Section: Minimal Nilpotent Hessenberg Varieties in Equivariant Geometry} seeks to introduce $X_H(e_{\theta})$ through the lens of equivariant geometry. Specifically, \ref{Section: Algebraic Group Actions} shows $X_H(e_{\theta})$ to be $T$-invariant, and it gives a description of the $T$-fixed point set $X_H(e_{\theta})^T$. Using this description, \ref{Section: Euler Number} computes $\vert X_H(e_{\theta})^T\vert$, the Euler number of $X_H(e_{\theta})$. Also, \ref{Section: The Size} uses properties of $X_H(e_{\theta})^T$ to give an upper bound on the codimension of $X_H(e_{\theta})$ in $G/B$. 

Section \ref{Section: Poincare Polynomials and Irreducible Components} exploits combinatorial descriptions of Hessenberg varieties in type $A$ to investigate the geometry of $X_H(e_{\theta})$. Specifically, \ref{Section: Poincare Polynomials in Type A} computes the Poincar\'{e} polynomial of $X_H(e_{\theta})$ by means of the Hessenberg stair shape diagram. Next, beginning with some partial results in type $ADE$, \ref{Subsection: Irreducible Components in Type ADE} and \ref{Subsection: Complete Description of the Irreducible Components in Type $A$} introduce the modified Hessenberg stair shape to completely describe the irreducible components of $X_H(e_{\theta})$ in type $A$.  

Section \ref{Section: GKM Theory on Minimal Nilpotent Hessenberg Varieties} studies $X_H(e_{\theta})$ via GKM theory. In \ref{Section: Brief Review of GKM Theory} and \ref{Section: The GKM Graph of G/B}, we review the relevant parts of this theory, as well as how they apply to the flag variety. Section \ref{Section: The GKM Graph} then shows the GKM graph of $X_H(e_{\theta})$ to be a full subgraph of the GKM graph of $G/B$. In \ref{Section: GKM Graphs in Type A}, we explain how one would implement this result to draw the GKM graph of $X_H(e_{\theta})$ in type $A$. We then provide the GKM graphs of all five such Hessenberg varieties in type $A_2$.   

Section \ref{Section: Cohomology Ring Presentations} is devoted to the calculation of $H^*(X_H(e_{\theta});\mathbb{Q})$ and $H_T^*(X_H(e_{\theta});\mathbb{Q})$. Specifically, the restriction maps $H^*(G/B;\mathbb{Q})\rightarrow H^*(X_H(e_{\theta});\mathbb{Q})$ and $H_T^*(G/B;\mathbb{Q})\rightarrow H_T^*(X_H(e_{\theta});\mathbb{Q})$ are shown to be surjective with kernels generated by certain opposite Schubert classes.

\vspace{10pt}

\noindent\emph{\textbf{Acknowledgements.}}  We gratefully acknowledge Megumi Harada and Julianna Tymoczko for their having organized the Workshop on Recent Developments in the Geometry and Combinatorics of Hessenberg Varieties at the Fields Institute in July 2015. We also wish to thank the conference participants, particularly Julianna Tymoczko and Martha Precup, for enlightening conversations. Additionally, the first author thanks Yoshihiro Ohnita and Mikiya Masuda for enhancing his research projects, and the second author is grateful to Lisa Jeffrey and John Scherk for the same reason.  The first author was supported by the JSPS Program for Advancing Strategic International Networks to Accelerate 
the Circulation of Talented Researchers:
``Mathematical Science of Symmetry, Topology and Moduli,
Evolution of International Research Network based on OCAMI''. He is also supported by a JSPS Grant-in-Aid for Young Scientists (B): 15K17544. The second author was supported by an OGS Queen Elizabeth II scholarship during the preparation of this manuscript.

\section{Background}\label{Section: Background}
\subsection{Notation and Conventions}\label{Section: Notation}
We begin by introducing some of the objects that will remain fixed throughout the article. Let $G$ denote a connected, simply-connected simple algebraic group over $\mathbb{C}$ with Lie algebra $\mathfrak{g}$. One then has the adjoint representations $$\Adj:G\rightarrow\GL(\mathfrak{g}),\text{ }\text{ }g\mapsto\Adj_g$$ and $$\adj:\mathfrak{g}\rightarrow\mathfrak{gl}(\mathfrak{g}),\text{ }\text{ }\xi\mapsto\adj_{\xi}.$$ 

Fix a pair of opposite Borel subgroups $B,B_{-}\subseteq G$, whose intersection is then a maximal torus $T=B\cap B_{-}$. Let $\mathfrak{t}$ and $\mathfrak{b}$ denote the Lie algebras of $T$ and $B$, respectively. One has the weight lattice $X^*(T):=\Hom(T,\mathbb{C}^*)$ and collections of roots $\Delta\subseteq X^*(T)$, positive roots $\Delta_{+}\subseteq\Delta$, negative roots $\Delta_{-}\subseteq\Delta$, and simple roots $\Pi\subseteq\Delta_{+}$. Since $G$ is simple, there is a unique highest root $\theta\in\Delta_{+}$. Finally, let $W=N_G(T)/T$ be the Weyl group.

By virtue of the choices made above, $\Delta$ and $W$ are posets. The partial order on the former is $$\beta\leq\gamma\Longleftrightarrow \gamma-\beta=\sum_{\alpha\in\Pi}n_{\alpha}\alpha$$ for some $n_{\alpha}\in\mathbb{Z}_{\geq 0}$. The Weyl group $W$ carries the Bruhat order, an excellent reference for which is \cite{DeodharCoxeter}. We shall always assume $\Delta$ and $W$ to be partially ordered in the manner described here.

This article will make extensive use of $T$-equivariant cohomology. More explicitly, let $ET\rightarrow BT$ denote the universal principal $T$-bundle. If $X$ is a topological space equipped with a continuous $T$-action, then the product $ET\times X$ carries the diagonal $T$-action. Furthermore, the \textit{Borel mixing space} $X_T$ of $X$ is defined by $$X_T:=(ET\times X)/T.$$ The $T$-\textit{equivariant cohomology} of $X$ (with rational coefficients) is then defined to be $$H_T^*(X;\mathbb{Q}):=H^*(X_T;\mathbb{Q}),$$ the ordinary cohomology of $X_T$. We will henceforth assume all cohomology (both ordinary and equivariant) and homology to be over $\mathbb{Q}$, and we will suppress the $\mathbb{Q}$-coefficients in our notation.

Let us take a moment to recall an important description of $H_T^*(\text{pt})$, the $T$-equivariant cohomology of the one-point space. Given $\alpha\in X^*(T)$, let $\mathbb{C}_{\alpha}$ denote the one-dimensional $T$-representation of weight $\alpha$. We may regard $\mathbb{C}_{\alpha}$ as a $T$-equivariant complex line bundle over a point. As such, it has a $T$-equivariant first Chern class, $c_1^T(\mathbb{C}_{\alpha})\in H_T^2(\text{pt})$. We then have a degree-doubling $\mathbb{Q}$-algebra isomorphism $$\varphi:\Sym(X^*(T)\otimes_{\mathbb{Z}}\mathbb{Q})\rightarrow H_T^*(\text{pt}),$$ characterized by the property that $\varphi(\alpha)=c_1^T(\mathbb{C}_{\alpha})$ for all $\alpha\in X^*(T)$. With this in mind, we will make no distinction between $H_T^*(\text{pt})$ and $\Sym(X^*(T)\otimes_{\mathbb{Z}}\mathbb{Q})$.

\subsection{The Hessenberg Varieties of Interest}\label{Section: Minimal Nilpotent Hessenberg Variety} 
Suppose that $H\subseteq\mathfrak{g}$ is a \textit{Hessenberg subspace}, namely a $\mathfrak{b}$-invariant subspace of $\mathfrak{g}$ containing $\mathfrak{b}$.\footnote{We emphasize that $H$ need not be a parabolic subalgebra of $\mathfrak{g}$.} Note that
\begin{equation}\label{Definition: Hessenberg roots}
H=\mathfrak{t}\oplus\bigoplus_{\gamma\in\Delta_H}\mathfrak{g}_{\gamma}=\mathfrak{b}\oplus\bigoplus_{\gamma\in\Delta_H^-}\mathfrak{g}_{\gamma}
\end{equation}
for some subsets $\Delta_H\subseteq\Delta$ and $\Delta_{H}^{-}\subseteq\Delta_{-}$ where $\Delta_{H}^{-}=\Delta_{H}\cap \Delta_{-}$.
We shall call the roots in $\Delta_H$ \textit{Hessenberg roots}, while calling those in $\Delta_H^-$ \textit{negative Hessenberg roots}.

Now, given $\xi\in\mathfrak{g}$, the subset $$G_H(\xi):=\{g\in G:\Adj_{g^{-1}}(\xi)\in H\}$$ is invariant under the right-multiplicative action of $B$ on $G$. We may therefore define $$X_H(\xi):=G_H(\xi)/B.$$ This is a closed (hence projective) subvariety of $G/B$, called a \textit{Hessenberg variety} (see \cite{DeMari}). If $\xi\in\mathfrak{g}$ is nilpotent, one calls $X_H(\xi)$ a \textit{nilpotent Hessenberg variety}.

The following relationship between adjoint orbits and Hessenberg varieties will help to give context for the Hessenberg varieties studied in this manuscript. 

\begin{lemma}\label{Conjugation Invariance}
If $\xi,\eta\in\mathfrak{g}$ belong to the same $G$-orbit, then $X_H(\xi)$ and $X_H(\eta)$ are isomorphic as varieties.
\end{lemma}

\begin{proof}
By assumption $\eta=\Adj_g(\xi)$ for some $g\in G$. Note that left-multiplication by $g$ defines an isomorphism from $G_H(\xi)$ to $G_H(\eta)$. This isomorphism is $B$-equivariant for the right-multiplicative action of $B$. Hence, the quotients $G_H(\xi)/B=X_H(\xi)$ and $G_H(\eta)/B=X_H(\eta)$ are isomorphic. 
\end{proof}

Fix a non-zero vector in the highest root space, $e_{\theta}\in\mathfrak{g}_{\theta}\setminus\{0\}$, and consider the nilpotent Hessenberg variety $X_H(e_{\theta})$. Noting that $e_{\theta}$ belongs to the minimal nilpotent orbit $\mathcal{O}_{\text{min}}$ of $G$, Lemma \ref{Conjugation Invariance} implies that $X_H(\xi)\cong X_H(e_{\theta})$ for all $\xi\in\mathcal{O}_{\text{min}}$. In this sense, $X_H(e_{\theta})$ is precisely the Hessenberg variety arising from the minimal nilpotent orbit. 

Letting the Hessenberg subspace $H$ vary, the $X_H(e_{\theta})$ constitute an interesting family of subvarieties of $G/B$. With respect to inclusion, the largest and smallest are $X_{\mathfrak{g}}(e_{\theta})$ and $X_{\mathfrak{b}}(e_{\theta})$, respectively. The former is easily seen to be $G/B$ itself, while the latter is the Springer fibre above $e_{\theta}$. In particular, $X_H(e_{\theta})$ is sometimes singular and reducible.

To obtain additional examples, we will need to recall a concrete description of Hessenberg varieties in type $A$.
  
\subsection{Examples in Type $A$}\label{Section: Examples in Type A}

Suppose that $G=\SL_n(\mathbb{C})$ with $n\geq2$, and that $T$ are $B$ are the subgroups of diagonal and upper-triangular matrices in $\SL_n(\mathbb{C})$, respectively. For distinct $i,j\in\{1,2,\ldots,n\}$, $t_i-t_j$ shall denote the root \begin{equation}
T\rightarrow\mathbb{C}^*,\text{ }\text{ }\begin{bmatrix} t_1 & 0 & 0 & \ldots & 0\\ 0 & t_2 & 0 & \ldots & 0\\ \vdots & \vdots & \vdots & \ddots & \vdots\ \\ 0 & 0 & 0 & \ldots & t_n\end{bmatrix}\mapsto t_it_{j}^{-1}.\end{equation} The highest root is then given by \begin{equation}\label{Highest Root}
\theta:=t_1-t_n,\end{equation} and \begin{align*}
e_{\theta} := 
 \begin{bmatrix}
  0 & 0 & \ldots & 0 & 1 \\
  0 & 0 & \ldots & 0 & 0\\
  \vdots & \vdots & \ddots & \vdots & \vdots \\
  0 & 0 & \ldots & 0 & 0
 \end{bmatrix}
\end{align*} is a choice of highest root vector. 

Now, suppose that $H\subseteq\mathfrak{sl}_n(\mathbb{C})$ is a Hessenberg subspace. There exists a unique weakly increasing function $h:\{1,2,\ldots,n\}\rightarrow\{1,2,\ldots,n\}$ with $j\leq h(j)$ for all $j$, such that \begin{equation}\label{Hessenberg Correspondence}
H=\{[a_{ij}]\in\mathfrak{sl}_n(\mathbb{C}):a_{ij}=0\text{ for } i>h(j)\}.\end{equation} Noting that \eqref{Hessenberg Correspondence} defines a bijective correspondence between the Hessenberg subspaces $H$ and all such functions $h$, one calls these functions \textit{Hessenberg functions}. We will represent a Hessenberg function $h$ by listing its values, so that $h=(h(1),h(2),\ldots,h(n))$.

Consider the variety $Flags(\mathbb{C}^n)$ of full flags $V_{\bullet}=(\{0\}\subseteq V_1\subseteq V_2\subseteq\ldots\subseteq V_{n-1}\subseteq\mathbb{C}^n)$ of subspaces of $\mathbb{C}^n$. One has the usual variety isomorphism\begin{equation}\label{Flag Variety}
\SL_n(\mathbb{C})/B\cong Flags(\mathbb{C}^n).\end{equation} If $h=(h(1),h(2),\ldots,h(n))$ is the Hessenberg function corresponding to $H\subseteq\mathfrak{sl}_n(\mathbb{C})$, then \begin{equation}\label{type A example 50}X_H(e_{\theta})\cong\{V_{\bullet}\in Flags(\mathbb{C}^n):e_{\theta}(V_j)\subseteq V_{h(j)}\text{ for all }j\}\end{equation} via the isomorphism \eqref{Flag Variety}.

Let us use \eqref{type A example 50} to describe $X_H(e_{\theta})$ in the case where $n=3$ and $h=(2,3,3)$. Indeed, we have $$X_H(e_{\theta})\cong \{V_{\bullet}\in Flags(\mathbb{C}^3):e_{\theta}(V_1)\subseteq V_{2}\}.$$ Letting $\{e_1,e_2,e_3\}$ denote the standard basis of $\mathbb{C}^3$, it is straightforward to see that each $V_{\bullet}\in X_H(e_{\theta})$ must satisfy $V_1\subseteq\text{span}\{e_1,e_2\}$ or $e_1\in V_2$. Let $X_1$ and $X_2$ denote the subvarieties of $X_H(e_{\theta})$ defined by these respective conditions, so that $X_H(e_{\theta})=X_1\cup X_2$. Note that completing $V_1\subseteq\text{span}\{e_1,e_2\}$ to an element $V_{\bullet}\in X_{H}(e_{\theta})$ is equivalent to specifying a $2$-dimensional subspace $V_2$ containing $V_1$. Also, completing a $V_2$ containing $e_1$ to $V_{\bullet}\in X_{H}(e_{\theta})$ amounts to specifying a $1$-dimensional subspace $V_1$ contained in $V_2$. From these observations, we see that each of $X_1$ and $X_2$ is isomorphic to $\mathbb{P}^1\times\mathbb{P}^1$. The intersection of these subvarieties is seen to be two copies of $\mathbb{P}^1$ which themselves intersect in a single point. 

In Section \ref{Section: GKM Graphs in Type A}, we will study the above-mentioned example as a GKM variety (see Figure \ref{fig:GKM Graph for h=(2,3,3)}).

\section{The Equivariant Geometry of $X_H(e_{\theta})$}\label{Section: Minimal Nilpotent Hessenberg Varieties in Equivariant Geometry}
\subsection{Algebraic Group Actions on $X_H(e_{\theta})$}\label{Section: Algebraic Group Actions}
In contrast to a general nilpotent Hessenberg variety, $X_H(e_{\theta})$ is a union of Schubert varieties. Equivalently, we have the following proposition (cf. \cite{Tymoczko-Dimensions}, Proposition 4.1).

\begin{proposition}\label{B-invariance}
The variety $X_H(e_{\theta})$ is invariant under the action of $B$ on $G/B$. 
\end{proposition}

\begin{proof}
It suffices to prove that $G_H(e_{\theta})$ is invariant under left-multiplication by elements of $B$. To this end, suppose that $b\in B$ and $g\in G_H(e_{\theta})$. We have $$\Adj_{(bg)^{-1}}(e_{\theta})=\Adj_{g^{-1}}(\Adj_{b^{-1}}(e_{\theta})).$$ Since $e_{\theta}$ belongs to the highest root space, $\Adj_{b^{-1}}(e_{\theta})$ is a scalar multiple of $e_{\theta}$. Hence, $\Adj_{g^{-1}}(\Adj_{b^{-1}}(e_{\theta}))$ is a scalar multiple of $\Adj_{g^{-1}}(e_{\theta})$, and therefore in $H$.
\end{proof}

As a consequence of Proposition \ref{B-invariance}, $X_H(e_{\theta})$ carries an action of the maximal torus $T$. Properties of this $T$-action will play an essential role in proving the main results in this paper. The first such property is a description of the $T$-fixed point set $X_H(e_{\theta})^T$. To this end, recall that the $T$-fixed points of $G/B$ are enumerated by the bijection \begin{equation}\label{T-Fixed Points in G/B}W\xrightarrow{\cong} (G/B)^T\end{equation} $$w\mapsto x_{w}:=[g],$$ where $g\in N_G(T)$ represents $w\in W$. Recalling the definition of the Hessenberg roots $\Delta_H$ (see \eqref{Definition: Hessenberg roots}), we have the following description of $X_H(e_{\theta})^T$.

\begin{proposition}\label{Fixed Points}
The $T$-fixed points of $X_H(e_{\theta})$ are given by $$X_H(e_{\theta})^T=\{x_w:w\in W\text{ and }w^{-1}\theta\in \Delta_H\}.$$
\end{proposition}

\begin{proof}
Suppose that $w\in W$ is represented by $g\in N_G(T)$, so that $x_w=[g]\in G/B$. Hence, $x_w\in X_H(e_{\theta})$ if and only if $g=hb$ for some $h\in G_H(e_{\theta})$ and $b\in B$. Since $G_H(e_{\theta})$ is invariant under right-multiplication by elements of $B$, this is equivalent to the condition that $g\in G_H(e_{\theta})$. Equivalently, $\Adj_{g^{-1}}(e_{\theta})\in H$, which is precisely the statement that $w^{-1}\theta\in\Delta_H$. 
\end{proof}

For example, suppose that $G=\SL_n(\mathbb{C})$ and that $T\subseteq\SL_n(\mathbb{C})$ and $B\subseteq\SL_n(\mathbb{C})$ are the maximal torus and Borel subgroup considered in \ref{Section: Examples in Type A}, respectively. Recall that $W=S_n$ and let $h:\{1,2,\ldots,n\}\rightarrow\{1,2\ldots,n\}$ be a Hessenberg function corresponding to $H\subseteq\mathfrak{sl}_n(\mathbb{C})$. 
Since the highest root is as given in \eqref{Highest Root}, Proposition \ref{Fixed Points} implies that $x_w\in X_H(e_{\theta})^T$ if and only if the root space of $w^{-1}\theta:=t_{w^{-1}(1)}-t_{w^{-1}(n)}$ belongs to $H$. Noting that this root space is spanned by the matrix with entry $1$ in position $(w^{-1}(1),w^{-1}(n))$ and all other entries $0$, our characterization becomes \begin{equation}
\label{Type A Fixed Points}
x_w\in X_H(e_{\theta})^T\Longleftrightarrow w^{-1}(1)\leq h(w^{-1}(n)).\end{equation} 
This condition can be visualized in terms of the \textit{one-line notation}\footnote{For a permutation $w\in S_n$, the list $w(1) \ w(2) \ \ldots \ w(n)$ is called the one-line-notation for $w$.} for $w\in S_n$. Letting $i$ and $j$ be the positions of $1$ and $n$ in the one-line notation for $w$ respectively, the condition \eqref{Type A Fixed Points} becomes equivalent to $i\leq h(j)$ (ie. $h$ determines the amount by which the position of $1$ can exceed that of $n$ in the one-line notation). 

\subsection{The Euler Number of $X_H(e_{\theta})$}\label{Section: Euler Number} 
This section addresses the computation of $\vert X_H(e_{\theta})^T\vert$, the Euler number of $X_H(e_{\theta})$. More precisely, we give a general formula for $\vert X_H(e_{\theta})^T\vert$ and then specialize it to cases in which one can be more explicit. 

To begin, Proposition \ref{Fixed Points} implies that \begin{equation}\label{Euler 1}\vert X_H(e_{\theta})^T\vert=\vert\{w\in W:w^{-1}\theta\in\Delta_H\}\vert=\vert\{w\in W:w\theta\in\Delta_H\}\vert.\end{equation} Since $W$ preserves the set of long roots $\Delta_{\text{long}}\subseteq\Delta$, \eqref{Euler 1} becomes \begin{equation}\label{Euler 2}
\vert X_H(e_{\theta})^T\vert=\vert\{w\in W:w\theta\in\Delta_{\text{long,H}}\}\vert,
\end{equation} where $\Delta_{\text{long},H}:=\Delta_{\text{long}}\cap\Delta_H$. Noting that $W$ acts transitively on $\Delta_{\text{long}}$, we have $$\vert\{w\in W:w\theta=\alpha\}\vert=\vert\{w\in W:w\theta=\theta\}\vert=\frac{\vert W\vert}{\vert\Delta_{\text{long}}\vert}$$ for all $\alpha\in\Delta_{\text{long},H}$. 
Hence, \eqref{Euler 2} becomes 
\begin{proposition}\label{Prop General Euler Number}
\begin{equation}\label{General Euler Number}
\vert X_H(e_{\theta})^T\vert=\vert W\vert\frac{\vert\Delta_{\emph{long},H}\vert}{\vert\Delta_{\emph{long}}\vert}.
\end{equation}   
\end{proposition}

We now specialize \eqref{General Euler Number} to some particularly tractable cases. Firstly, \eqref{General Euler Number} is seen to imply that the Springer fibre $X_{\mathfrak{b}}(e_{\theta})$ contains exactly one-half of the $T$-fixed points in $G/B$.

\begin{corollary}\label{Euler Number of Springer Fibre}
The Euler number of our Springer fibre $X_{\mathfrak{b}}(e_{\theta})$ is given by 
\begin{equation}\label{Euler Number of Springer Fibre}
\vert X_{\mathfrak{b}}(e_{\theta})^T\vert=\frac{\vert W\vert}{2}.
\end{equation}
\end{corollary}

\begin{proof}
Since the number of positive long roots coincides with the number of negative long roots, we see that $|\Delta_{\text{long},\mathfrak{b}}|=\frac{1}{2}|\Delta_{\text{long}}|$. The formula \eqref{Euler Number of Springer Fibre} then follows from \eqref{General Euler Number}. 
\end{proof}

Our second specialization of \eqref{General Euler Number} is to the simply-laced case, in which $\Delta_{\text{long}}=\Delta$ and $\Delta_{\text{long},H}=\Delta_H$. Hence, $\vert\Delta_{\text{long}}\vert=\dim(\mathfrak{g})-\text{rank}(\mathfrak{g})$ and $\vert\Delta_{\text{long},H}\vert=\dim(H)-\text{rank}(\mathfrak{g})$, so that \eqref{General Euler Number} reads as 
\begin{corollary}
In the simply-laced case, we have
\begin{equation}\label{Simply-Laced Euler Number}
\vert X_H(e_{\theta})^T\vert=\vert W\vert\left(\frac{\dim(H)-\emph{rank}(\mathfrak{g})}{\dim(\mathfrak{g})-\emph{rank}(\mathfrak{g})}\right).
\end{equation}
\end{corollary}
For example, if $G=\SL_n(\mathbb{C})$, then $$\vert X_H(e_{\theta})^T\vert=(n-2)!(\dim(H)-n+1).$$

\subsection{The Codimension of $X_H(e_{\theta})$}\label{Section: The Size}
Let us write $\Pi=\{\alpha_1,\alpha_2,\ldots,\alpha_n\}$, and let $\Pi_i:=\Pi-\{\alpha_i\}$ for $i\in\{1,2,\ldots,n\}$.
Denote by $P_i\subseteq G$ the \textit{maximal} parabolic subgroup corresponding to $\Pi_i\subseteq\Pi$. 
Also, let $W_{P_i}\subseteq W$ be the subgroup generated by the simple reflections $s_{\alpha_k}$ for $k\neq i$.
\begin{lemma}\label{Maximal Parabolic}
If $w\in W_{P_i}$, then $w^{-1}\theta\in\Delta_+$.
\end{lemma}
\begin{proof}
Note that $\theta=\sum_{k=1}^n m_k\alpha_k$ with $m_k>0$ for $k\in\{1,2,\ldots,n\}$.
Also, recall that we have 
\begin{equation}\label{Simple Reflection}
s_{\alpha_{\ell}}(\alpha_{k}) = \alpha_{k} - \frac{2\langle \alpha_{k}, \alpha_{\ell}\rangle}{\langle \alpha_{\ell}, \alpha_{\ell}\rangle}\alpha_{\ell}.
\end{equation}
If $w\in W_{P_i}$ and we write $w^{-1}\theta=\sum_{k=1}^n d_k\alpha_k$ for some $d_k\in\mathbb{Z}$, then \eqref{Simple Reflection} implies $d_i=m_i>0$. Since $w^{-1}\theta$ is a root, this shows it to be a positive root.
\end{proof}

\begin{proposition}\label{thickness}
For any maximal parabolic subgroup $P_i$, we have
$P_i/B \subseteq X_H(e_{\theta}) \subseteq G/B$.
\end{proposition}
\begin{proof}
Lemma \ref{Maximal Parabolic} implies that $w^{-1}\theta\in\Delta_H$ for any $w\in W_{P_i}$, which by Proposition \ref{Fixed Points} means that $x_w\in X_H(e_{\theta})^T$.
Since we have $P_i/B=\coprod_{w\in W_{P_i}}BwB/B$ and $X_H(e_{\theta})=\coprod_{x_w\in X_H(e_{\theta})^T}BwB/B$, we obtain $P_i/B \subseteq X_H(e_{\theta}) \subseteq G/B$.
\end{proof}

This proposition has interesting implications for the codimension of $X_H(e_{\theta})$ in $G/B$.
Indeed, when $G=\SL_n(\mathbb{C})$, a suitable choice of maximal parabolic $P_i$ gives (by Proposition \ref{thickness}) $$Flags(\mathbb{C}^{n-1}) \subseteq X_H(e_{\theta}) \subseteq Flags(\mathbb{C}^n).$$
Hence, the complex codimension of $X_H(e_{\theta})$ in $Flags(\mathbb{C}^n)$ is at most $n-1$ when $G=\SL_n(\mathbb{C})$. Finally, in all Lie types, it is known that the codimension of $X_{\mathfrak{b}}(e_{\theta})$ in $G/B$ is equal to half the dimension of the minimal nilpotent orbit $\mathcal{O}_{\text{min}}$ (cf. \cite{Chriss}, Corollary 3.3.24, which is based on \cite{Spaltenstein}). Since $\dim_{\mathbb{C}}(\mathcal{O}_{\text{min}})=2h^{\vee}-2$ (see \cite{Wang}, Theorem 1)\footnote{Here, $h^{\vee}$ is the \textit{dual Coxeter number}.}, we have 
\[
\text{codim}_{\mathbb{C}}(X_{\mathfrak{b}}(e_{\theta}))=h^{\vee}-1.
\]
For a general Hessenberg subspace $H$, the inclusion $X_{\mathfrak{b}}(e_{\theta})\subseteq X_H(e_{\theta})$ gives
\[
\text{codim}_{\mathbb{C}} (X_H(e_{\theta})) \leq h^{\vee}-1.
\]

\section{Poincar\'{e} Polynomials and Irreducible Components}\label{Section: Poincare Polynomials and Irreducible Components}

\subsection{Poincar\'{e} Polynomials in Type $A$}\label{Section: Poincare Polynomials in Type A}
We now compute the Poincar\'{e} polynomial $P_H(t)$ of $X_H(e_{\theta})$ when $G=\SL_n(\mathbb{C})$. Accordingly, we shall assume all notation to be as in \ref{Section: Examples in Type A}.

Consider an $n\times n$ grid of boxes, and let $(i,j)$ denote the box in row $i$ and column $j$. If $H\subseteq\mathfrak{sl}_n(\mathbb{C})$ is a Hessenberg subspace with Hessenberg function $h$, we shall call the stair-shaped sub-grid $$\{(i,j)\in\{1,2,\ldots,n\}\times\{1,2,\ldots,n\}\}:i\leq h(j)\}$$ the \textit{Hessenberg stair shape}. One identifies it by drawing a line in the $n\times n$ grid such that the sub-grid consists precisely of the boxes lying above the line. 

\begin{figure}[h]
\begin{picture}(350,90)(0,-90)
\put(125,0){\line(0,-1){75}}
\put(125,0){\line(1,0){15}}
\put(125,-15){\line(1,0){15}}
\put(125,-30){\line(1,0){15}}
\put(125,-45){\line(1,0){15}}
\put(125,-60){\line(1,0){15}}
\put(125,-75){\line(1,0){15}}
\put(140,0){\line(0,-1){75}}
\put(140,0){\line(1,0){15}}
\put(140,-15){\line(1,0){15}}
\put(140,-30){\line(1,0){15}}
\put(140,-45){\line(1,0){15}}
\put(140,-60){\line(1,0){15}}
\put(140,-75){\line(1,0){15}}
\put(155,0){\line(0,-1){75}}
\put(155,0){\line(1,0){15}}
\put(155,-15){\line(1,0){15}}
\put(155,-30){\line(1,0){15}}
\put(155,-45){\line(1,0){15}}
\put(155,-60){\line(1,0){15}}
\put(155,-75){\line(1,0){15}}
\put(170,0){\line(0,-1){75}}
\put(170,0){\line(1,0){15}}
\put(170,-15){\line(1,0){15}}
\put(170,-30){\line(1,0){15}}
\put(170,-45){\line(1,0){15}}
\put(170,-60){\line(1,0){15}}
\put(170,-75){\line(1,0){15}}
\put(185,0){\line(0,-1){75}}
\put(185,0){\line(1,0){15}}
\put(185,-15){\line(1,0){15}}
\put(185,-30){\line(1,0){15}}
\put(185,-45){\line(1,0){15}}
\put(185,-60){\line(1,0){15}}
\put(185,-75){\line(1,0){15}}
\put(200,0){\line(0,-1){75}}
\linethickness{0.5mm}
\put(125,0){\line(0,-1){30}}
\put(125,-30){\line(1,0){15}}
\put(140,-30){\line(0,-1){30}}
\put(140,-60){\line(1,0){15}}
\put(155,-60){\line(0,-1){15}}
\put(155,-75){\line(1,0){45}}
\end{picture}
\caption{The Hessenberg stair shape determined by $h=(2,4,5,5,5)$ when $n=5$}
\end{figure}
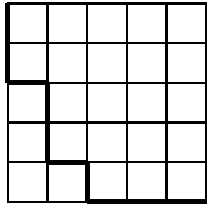

\begin{definition}\label{Definition of Diagonals}
For $0\leq i\leq 2n-3$, we define $q_H(i)$ to be the number of boxes in the Hessenberg stair shape meeting the diagonal line segment joining $(2,n-i)$ and $(2+i,n)$.
Namely, 
\begin{align*}
q_H(i) 
= |\{ (k,j)\in\{1,2,\ldots,n\}\times\{1,2,\ldots,n\}: 2\leq k\leq h(n+1-j), \ j+k-3=i \}|.
\end{align*}
\end{definition}
\vspace{10pt}

The number $q_H(i)$ is easily computed in practice. One starts with the rightmost box in row \#2, moves $i$ boxes to the left, and then draws the longest possible diagonal line segment passing through the current box and not passing through a box in row \#1. The number of boxes meeting this line segment is precisely $q_H(i)$.

The following figure illustrates the computation of $q_H(2)$ in the case $n=5$ and $h=(2,4,5,5,5)$.
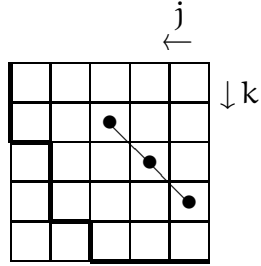
\begin{figure}[h]
\begin{picture}(350,90)(0,-90)
\put(125,0){\line(0,-1){75}}
\put(125,0){\line(1,0){15}}
\put(125,-15){\line(1,0){15}}
\put(125,-30){\line(1,0){15}}
\put(125,-45){\line(1,0){15}}
\put(125,-60){\line(1,0){15}}
\put(125,-75){\line(1,0){15}}
\put(140,0){\line(0,-1){75}}
\put(140,0){\line(1,0){15}}
\put(140,-15){\line(1,0){15}}

\put(140,-30){\line(1,0){15}}
\put(140,-45){\line(1,0){15}}
\put(140,-60){\line(1,0){15}}
\put(140,-75){\line(1,0){15}}
\put(155,0){\line(0,-1){75}}
\put(155,0){\line(1,0){15}}
\put(155,-15){\line(1,0){15}}
\put(155,-30){\line(1,0){15}}
\put(155,-45){\line(1,0){15}}
\put(155,-60){\line(1,0){15}}
\put(155,-75){\line(1,0){15}}
\put(170,0){\line(0,-1){75}}
\put(170,0){\line(1,0){15}}
\put(170,-15){\line(1,0){15}}
\put(170,-30){\line(1,0){15}}
\put(170,-45){\line(1,0){15}}
\put(170,-60){\line(1,0){15}}
\put(170,-75){\line(1,0){15}}
\put(185,0){\line(0,-1){75}}
\put(185,0){\line(1,0){15}}
\put(185,-15){\line(1,0){15}}
\put(185,-30){\line(1,0){15}}
\put(185,-45){\line(1,0){15}}
\put(185,-60){\line(1,0){15}}
\put(185,-75){\line(1,0){15}}
\put(200,0){\line(0,-1){75}}
\linethickness{0.5mm}
\put(125,0){\line(0,-1){30}}
\put(125,-30){\line(1,0){15}}
\put(140,-30){\line(0,-1){30}}
\put(140,-60){\line(1,0){15}}
\put(155,-60){\line(0,-1){15}}
\put(155,-75){\line(1,0){45}}
\put(162.5,-22.5){\circle*{5}}
\put(177.5,-37.5){\circle*{5}}
\put(192.5,-52.5){\circle*{5}}
\put(162,-22.5){\line(1,-1){30}}
\put(182,5){$\leftarrow$}
\put(187,15){$j$}
\put(204,-15){$\downarrow$}
\put(212,-15){$k$}
\end{picture}
\caption{The computation of $q_H(2)$ when $h=(2,4,5,5,5)$ and $n=5$. As per Definition \ref{Definition of Diagonals}, one begins by drawing the diagonal line segment connecting $(2,3)$ and $(4,5)$. Since this segment meets exactly $3$ boxes, we have $q_H(2)=3$.}
\end{figure}

It will be convenient to consider the polynomial 
$$q_H(t):=\sum_{i=0}^{2n-3}q_H(i)t^{2i}.$$ As is the case with its coefficients, this polynomial can be computed diagrammatically via the Hessenberg stair shape. One simply fills each box involved in the computation of $q_H(i)$ with $t^{2i}$, and then sums the resulting terms. The following figure illustrates this procedure.

\begin{figure}[h]
	\begin{picture}(350,105)(0,-90)
	\put(125,0){\line(0,-1){75}}
	\put(125,0){\line(1,0){15}}
	\put(125,-15){\line(1,0){15}}
	\put(125,-30){\line(1,0){15}}
	\put(125,-45){\line(1,0){15}}
	\put(125,-60){\line(1,0){15}}
	\put(125,-75){\line(1,0){15}}
	\put(140,0){\line(0,-1){75}}
	\put(140,0){\line(1,0){15}}
	\put(140,-15){\line(1,0){15}}
	\put(140,-30){\line(1,0){15}}
	\put(140,-45){\line(1,0){15}}
	\put(140,-60){\line(1,0){15}}
	\put(140,-75){\line(1,0){15}}
	\put(155,0){\line(0,-1){75}}
	\put(155,0){\line(1,0){15}}
	\put(155,-15){\line(1,0){15}}
	\put(155,-30){\line(1,0){15}}
	\put(155,-45){\line(1,0){15}}
	\put(155,-60){\line(1,0){15}}
	\put(155,-75){\line(1,0){15}}
	\put(170,0){\line(0,-1){75}}
	\put(170,0){\line(1,0){15}}
	\put(170,-15){\line(1,0){15}}
	\put(170,-30){\line(1,0){15}}
	\put(170,-45){\line(1,0){15}}
	\put(170,-60){\line(1,0){15}}
	\put(170,-75){\line(1,0){15}}
	\put(185,0){\line(0,-1){75}}
	\put(185,0){\line(1,0){15}}
	\put(185,-15){\line(1,0){15}}
	\put(185,-30){\line(1,0){15}}
	\put(185,-45){\line(1,0){15}}
	\put(185,-60){\line(1,0){15}}
	\put(185,-75){\line(1,0){15}}
	\put(200,0){\line(0,-1){75}}
	\linethickness{0.5mm}
	\put(125,0){\line(0,-1){30}}
	\put(125,-30){\line(1,0){15}}
	\put(140,-30){\line(0,-1){30}}
	\put(140,-60){\line(1,0){15}}
	\put(155,-60){\line(0,-1){15}}
	\put(155,-75){\line(1,0){45}}
	\put(182,5){$\leftarrow$}
	\put(187,15){$j$}
	\put(204,-15){$\downarrow$}
	\put(212,-15){$k$}
	\put(189,-26){$t^0$}
	\put(189,-41){$t^2$}
	\put(189,-56){$t^4$}
	\put(189,-71){$t^6$}
	\put(174,-26){$t^2$}
	\put(174,-41){$t^4$}
	\put(174,-56){$t^6$}
	\put(174,-71){$t^8$}
	\put(159,-26){$t^4$}
	\put(159,-41){$t^6$}
	\put(159,-56){$t^8$}
	\put(157,-71){$t^{10}$}
	\put(144,-26){$t^6$}
	\put(144,-41){$t^8$}
	\put(142,-56){$t^{10}$}
	\put(129,-26){$t^{8}$}
	\end{picture}
	\caption{The computation of $q_H(t)$ for $n=5$ and $h=(2,4,5,5,5)$. One has $q_H(t)=1+2t^2+3t^4+4t^6+4t^8+2t^{10}$.}
\end{figure}
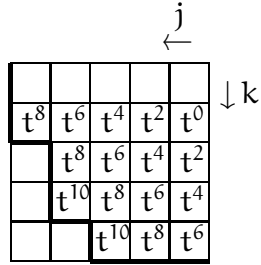

While $q_H(t)$ is not itself the  Poincar\'{e} polynomial of $X_H(e_{\lambda})$, we have the following proposition. 

\begin{proposition}\label{Poincare Polynomial Proposition}
The Poincar\'{e} polynomial $P_H(t)$ of $X_H(e_{\theta})$ is given by
\begin{align*}
P_H(t) = q_H(t)\cdot \prod_{\ell=1}^{n-3}(1+t^2+\ldots+t^{2\ell}).
\end{align*}
\end{proposition}
\begin{proof}
Since $X_H(e_{\theta})$ is a union of Schubert cells, the type $A$ fixed point criterion \eqref{Type A Fixed Points} implies that
\begin{align*}
P_H(t) = \sum_{\substack{w\in S_n \\ w^{-1}(1)\leq h(w^{-1}(n))}} t^{2\ell(w)}. 
\end{align*}
Writing $j=w^{-1}(n)$ and $k=w^{-1}(1)$, we have
\begin{align*}
P_H(t) 
&= \sum_{j=1}^n 
\Big( 
\sum_{k=1}^{j-1} \sum_{v\in S_{n-2}} 
t^{2\ell(v)+2(k-1+n-j)}
+ \sum_{k=j+1}^{h(j)}\sum_{v\in S_{n-2}} 
t^{2\ell(v)+2(k-1+n-j-1)}
\Big),
\end{align*}
which can be explained as follows. In the one-line notation for $w$, if the position of $1$ is to the left of the position of $n$ (so $1\leq k\leq j-1$), then 1 has $k-1$ inversion pairs and $n$ has $n-j$ inversion pairs. If the position of $1$ is to the right of the position of $n$ (so $j+1\leq k\leq h(j)$), then $1$ has $k-1$ inversion pairs (including the pair $(n,1)$) and $n$ has $n-j-1$ inversion pairs (except for the pair $(n,1)$, which is already counted).

Now, note that
\begin{align*}
\sum_{v\in S_{n-2}} 
t^{2\ell(v)} = \prod_{\ell=1}^{n-3}(1+t^2+\ldots+t^{2\ell}),
\end{align*}
as each polynomial is the Poincar\'{e} polynomial of $Flags(\mathbb{C}^{n-2})$.
Hence, a direct computation gives
\begin{align*}
P_H(t) 
&= \sum_{j=1}^n 
\sum_{k=2}^{h(j)} t^{2(k-2+n-j)}
\cdot \prod_{\ell=1}^{n-3}(1+t^2+\ldots+t^{2\ell}) \\
&= \sum_{j=1}^n 
\sum_{k=2}^{h(n+1-j)} t^{2(j+k-3)}
\cdot \prod_{\ell=1}^{n-3}(1+t^2+\ldots+t^{2\ell}),
\end{align*}
and the claim follows.
\end{proof}

Recall that for $n=5$ and $h=(2,4,5,5,5)$, we have $q_H(t)=1+2t^2+3t^4+4t^6+4t^8+2t^{10}$. 
In this case, Proposition \ref{Poincare Polynomial Proposition} yields the Poincar\'{e} polynomial 
\begin{align*}
P_H(t) & = (1+2t^2+3t^4+4t^6+4t^8+2t^{10})\cdot (1+t^2)(1+t^2+t^4)\\
& = 1+4t^2+9t^4+15t^6+20t^8+21t^{10}+16t^{12}+8t^{14}+2t^{16}.
\end{align*}

On another note, since $h=(1,2,\ldots,n)$ corresponds to the Hessenberg subspace $\mathfrak{b}$, we obtain the following specialization of Proposition \ref{Poincare Polynomial Proposition}.
\begin{corollary}
The Poincar\'{e} polynomial of our Springer fiber $X_{\mathfrak{b}}(e_{\theta})$ is given by
\begin{align*}
P_{\mathfrak{b}}(t) = (1+2t^2+3t^4+\ldots+(n-1)t^{2(n-2)}) \cdot\prod_{\ell=1}^{n-3}(1+t^2+\ldots+t^{2\ell}).
\end{align*}
\end{corollary}

Additionally, Proposition \ref{Poincare Polynomial Proposition} allows one to deduce the following combinatorial formula for $\dim_{\mathbb{C}}(X_H(e_{\theta}))$.
{\begin{corollary}
The dimension of $X_H(e_{\theta})$ is given by
$$\dim_{\mathbb{C}}(X_H(e_{\theta}))=\frac{1}{2}(n-1)(n-2)+\emph{max}\ \{ h(j)-j \mid j=1,\ldots,n \}.$$
\end{corollary}

\subsection{Irreducible Components in Type $ADE$}\label{Subsection: Irreducible Components in Type ADE}

We now examine the irreducible components of $X_H(e_{\theta})$. As one might expect, these are precisely the maximal Schubert varieties $X(w):=\overline{BwB/B}$ contained in $X_H(e_{\theta})$. 

\begin{lemma}\label{Lemma: Irreducible Components of Closed B-Invariant Subvarieties}
The irreducible components of $X_H(e_{\theta})$ are the Schubert varieties $X(w)$ for the maximal $w\in W$ satisfying $w^{-1}\theta\in\Delta_{H}$.
\end{lemma}

\begin{proof}
By Proposition 1.5 of \cite{Hartshorne}, our task is to prove the following two statements.
\begin{itemize}
\item[(i)] The variety $X_H(e_{\theta})$ is a union of the $X(w)$ for the maximal $w\in W$ satisfying $w^{-1}\theta\in\Delta_H$.
\item[(ii)] If $w_1\neq w_2$ are two such maximal elements, then neither $X(w_1)\subseteq X(w_2)$ nor $X(w_2)\subseteq X(w_1)$ holds.
\end{itemize}

Since $X_H(e_{\theta})$ is a union of Schubert varieties (see Proposition \ref{B-invariance}), one for each $T$-fixed point in $X_H(e_{\theta})$, Proposition \ref{Fixed Points} allows us to write \begin{equation}\label{Union of Schubert Varieties}
X_H(e_{\theta})=\bigcup_{w^{-1}\theta\in\Delta_H}X(w)
\end{equation}
Furthermore, as $u\leq v$ if and only if $X(u)\subseteq X(v)$, \eqref{Union of Schubert Varieties} still holds if the union is taken only over the maximal $w\in W$ satisfying $w^{-1}\theta\in\Delta_{H}$. Hence, (i) is true. Of course, the fact $u\leq v\Longleftrightarrow X(u)\subseteq X(v)$ also implies that (ii) is true.   
\end{proof}

Let us assume $G$ to be of type $ADE$. For $\beta\in\Delta$, Lemma 4.4 of \cite{Tymoczko-Dimensions} allows one to consider the unique maximal $w_{\beta}\in W$ for which $w_{\beta}^{-1}\theta=\beta$. In other words, 
\begin{align}\label{definition of w_beta}
w_{\beta}:=\text{max} \{w\in W\mid w^{-1}\theta=\beta\}.
\end{align}
Note that if $w^{-1}\theta\in\Delta_H$, then $w\leq w_{\beta}$ for some $\beta\in\Delta_H$ (e.g. take $\beta=w^{-1}\theta$). It follows that the maximal elements of $\{w\in W \mid w^{-1}\theta\in\Delta_H\}$ (the set discussed in Lemma \ref{Lemma: Irreducible Components of Closed B-Invariant Subvarieties}) are precisely the maximal elements of  
\begin{equation}\label{Equation: Definition of Omega_H}
\Omega_H:=\{w_{\beta}:\beta\in\Delta_H\}.
\end{equation}
Using Lemma \ref{Lemma: Irreducible Components of Closed B-Invariant Subvarieties}, it follows that the maximal elements of $\Omega_H$ label the irreducible components of $X_H(e_{\theta})$.
However, $\Omega_H$ may still contain non-maximal elements, and the determination of its maximal elements will involve a few properties of the $w_{\beta}$.
 
\begin{proposition}\label{Proposition: Same Sign}
Suppose that $\beta,\gamma\in\Delta$.
\begin{itemize}
\item[(i)] $\beta=\gamma\Longleftrightarrow w_{\beta}=w_{\gamma}$
\item[(ii)] If $\beta$ and $\gamma$ have the same sign, then $\beta\leq\gamma\Longleftrightarrow w_{\gamma}\leq w_{\beta}$.
\end{itemize}
\end{proposition}

\begin{proof}
For (i), the forward implication is clear. Conversely, if $w_{\beta}=w_{\gamma}$, then $\beta=w_{\beta}^{-1}\theta=w_{\gamma}^{-1}\theta=\gamma.$

For (ii), the result is explicitly stated in the proof of Proposition 4.5 of \cite{Tymoczko-Dimensions}. The result itself can be seen to follow from Proposition 3.2 of \cite{Stembridge}, together with the subword characterization of the Bruhat order.
\end{proof}

\begin{corollary}\label{Corollary: Reduction}
\begin{itemize}
\item[(i)] If $\alpha,\beta\in\Pi$ are distinct simple roots, then $w_{\alpha}$ and $w_{\beta}$ are incomparable in the Bruhat order.
\item[(ii)] If $\gamma$ and $\delta$ are distinct minimal elements of $\Delta_H^{-}$, then $w_{\gamma}$ and $w_{\delta}$ are incomparable in the Bruhat order.
\item[(iii)] Suppose that $\gamma\in\Delta_H$. If $w_{\gamma}$ is a maximal element of $\Omega_H$, then $\gamma\in\Pi$ or $\gamma$ is a minimal element of $\Delta_H^{-}$. 
\end{itemize}
\end{corollary}

\begin{proof}
Recognizing (i) and (ii) as immediate consequences of Proposition \ref{Proposition: Same Sign}, we prove only (iii). To this end, if $\gamma$ is positive, then there exists $\alpha\in\Pi$ such that $\alpha\leq\gamma$. Proposition \ref{Proposition: Same Sign} implies that $w_{\gamma}\leq w_{\alpha}$, and the maximality of $w_{\gamma}$ then yields $w_{\gamma}=w_{\alpha}$. It follows that $\gamma=\alpha$ is simple.

Now, assume that $\gamma$ is negative and let $\delta\in\Delta_H^{-}$ satisfy $\delta\leq\gamma$. Proposition \ref{Proposition: Same Sign} implies $w_{\gamma}\leq w_{\delta}$. Since $w_{\gamma}$ is maximal, $w_{\gamma}=w_{\delta}$ and we conclude that $\gamma=\delta$. It follows that $\gamma$ is a minimal element of $\Delta_{H}^{-}$. 
\end{proof}

In light of Corollary \ref{Corollary: Reduction}, the maximal elements of $\Omega_H$ are of the following two types: 
\begin{enumerate}
\item $w_{\alpha}$, where $\alpha\in\Pi$ and $w_{\alpha}\nless w_{\gamma}$ for all $\gamma\in\Delta_H^{-}$, \footnote{Strictly speaking, Corollary \ref{Corollary: Reduction} gives the following different-looking description of the maximal $w_{\alpha}$'s: $w_{\alpha}$, where $\alpha\in\Pi$ and $w_{\alpha}\nless w_{\gamma}$ for all minimal $\gamma\in\Delta_H^{-}$. However, by appealing to Proposition \ref{Proposition: Same Sign}, one sees that this is equivalent to the description we have given.}
\item $w_{\gamma}$, where $\gamma$ is a minimal element of $\Delta_H^{-}$ and $w_{\gamma}\nless w_{\alpha}$ for all $\alpha\in\Pi$.
\end{enumerate}

In order to refine (2), we will need the following two results.

\begin{lemma}\label{Lemma: Simple Inequality}
If $\alpha\in\Pi$, then $w_{\alpha}<w_{-\alpha}$.
\end{lemma}

\begin{proof}
Since $(w_{\alpha}s_{\alpha})^{-1}\theta=-\alpha$, the maximality of $w_{-\alpha}$ implies that \begin{equation}\label{First Inequality}w_{\alpha}s_{\alpha}\leq w_{-\alpha}.\end{equation} Furthermore, as $(w_{\alpha}s_{\alpha})\alpha=-\theta\in\Delta_{-}$, we have \begin{equation}\label{Second Inequality}w_{\alpha}<w_{\alpha}s_{\alpha}.\end{equation} By combining \eqref{First Inequality} and \eqref{Second Inequality}, we obtain the desired result.
\end{proof}

\begin{corollary}\label{Corollary: Not Less}
If $\alpha\in\Pi$ and $\gamma\in\Delta_{-}$, then $w_{\gamma}\nless w_{\alpha}$.
\end{corollary}

\begin{proof}
If $w_{\gamma}<w_{\alpha}$, then Lemma \ref{Lemma: Simple Inequality} implies that $w_{\gamma}<w_{-\alpha}$. Proposition \ref{Proposition: Same Sign} then yields $-\alpha<\gamma$, which is impossible.
\end{proof}

In light of the above, we have the following improved description of the maximal elements of $\Omega_H$:

\begin{enumerate}
	\item $w_{\alpha}$, where $\alpha\in\Pi$ and $w_{\alpha}\nless w_{\gamma}$ for all $\gamma\in\Delta_H^{-}$
	\item $w_{\gamma}$, where $\gamma$ is a minimal element of $\Delta_H^{-}$.
\end{enumerate}

Remembering that the maximal elements of $\Omega_H$ label the irreducible components of $X_H(e_{\theta})$, we have the following immediate corollary.

\begin{corollary}
If $\gamma$ is a minimal element of $\Delta_H^{-}$, then $X(w_{\gamma})$ is an irreducible component of $X_H(e_{\theta})$.
\end{corollary}

Using (1) and (2), the next section gives a combinatorial enumeration of the maximal elements of $\Omega_H$ (and therefore also the irreducible components of $X_H(e_{\theta})$) in Lie type $A_{n-1}$.

\subsection{Complete Description of the Irreducible Components in Type $A$}\label{Subsection: Complete Description of the Irreducible Components in Type $A$}

Let $G=\SL_n(\mathbb{C})$ and assume all notation to be as presented in \ref{Section: Examples in Type A} and \ref{Section: Poincare Polynomials in Type A}. For $\beta=t_i-t_j\in\Delta$, \eqref{Highest Root}, \eqref{Type A Fixed Points}, and \eqref{definition of w_beta}} imply that $w_{\beta}\in S_n$ is the longest permutation satisfying $w_{\beta}(i)=1$ and $w_{\beta}(j)=n$. 
If $\alpha=t_{j-1}-t_j$ is a simple root, we have $w_{\alpha}(j-1)=1$ and $w_{\alpha}(j)=n$, i.e. the one-line notation for $w_{\alpha}$ is
\[
w_{\alpha}=\cdots \ 1 \ n \ \cdots
\]
where $1$ is in the $(j-1)$-st position, $n$ is in the $j$-th position, and the rest of the ordered sequence $w_{\alpha}(1),\dots,w_{\alpha}(j-2), w_{\alpha}(j+1),\dots,w_{\alpha}(n)$ is given by $n-1,n-2,\dots,3,2$.
For $\gamma=t_k-t_{\ell} \ (k>\ell)$ a negative root, the one-line notation for $w_{\gamma}$ is 
\[w_{\gamma}=\cdots \ n \ \cdots \ 1 \ \cdots,
\]
where $n$ is in the $\ell$-th position and $1$ is in the $k$-th position. 

Continuing with our specialization to type $A_{n-1}$, we will need to introduce the \textit{modified Hessenberg function} and the \textit{modified Hessenberg stair shape}. To this end, let $h$ be the Hessenberg function corresponding to the Hessenberg subspace $H\subseteq\mathfrak{sl}_n(\mathbb{C})$. We define a function $\overline{h}:\{1,2,\dots,n\}\rightarrow \{1,2,\dots,n\}$ by
\begin{equation}\label{def of modified function}
\overline{h}(j) := 
\begin{cases}
h(j)-1(=j-1) \quad &\text{if $h(j-1)=j-1$ and $h(j)=j$}, \\
h(j) &\text{otherwise},
\end{cases}
\end{equation}
and we call $\overline{h}$  the \textit{modified Hessenberg function}.
Note that while $\overline{h}$ is weakly increasing, it might not be an honest Hessenberg function. 

As with a Hessenberg function, one can consider the stair-shaped sub-grid 
$$\{(i,j)\in\{1,2,\ldots,n\}\times\{1,2,\ldots,n\}\}:i\leq \overline{h}(j)\},$$ called the \textit{modified Hessenberg stair shape} (see Figure \ref{modified stair-shap}).  

\begin{figure}[h]
\begin{picture}(350,130)(0,-130)
\put(25,0){\line(0,-1){120}}
\put(40,0){\line(0,-1){120}}
\put(55,0){\line(0,-1){120}}
\put(70,0){\line(0,-1){120}}
\put(85,0){\line(0,-1){120}}
\put(100,0){\line(0,-1){120}}
\put(115,0){\line(0,-1){120}}
\put(130,0){\line(0,-1){120}}
\put(145,0){\line(0,-1){120}}
\put(25,0){\line(1,0){120}}
\put(25,-15){\line(1,0){120}}
\put(25,-30){\line(1,0){120}}
\put(25,-45){\line(1,0){120}}
\put(25,-60){\line(1,0){120}}
\put(25,-75){\line(1,0){120}}
\put(25,-90){\line(1,0){120}}
\put(25,-105){\line(1,0){120}}
\put(25,-120){\line(1,0){120}}
\put(225,0){\line(0,-1){120}}
\put(240,0){\line(0,-1){120}}
\put(255,0){\line(0,-1){120}}
\put(270,0){\line(0,-1){120}}
\put(285,0){\line(0,-1){120}}
\put(300,0){\line(0,-1){120}}
\put(315,0){\line(0,-1){120}}
\put(330,0){\line(0,-1){120}}
\put(345,0){\line(0,-1){120}}
\put(225,0){\line(1,0){120}}
\put(225,-15){\line(1,0){120}}
\put(225,-30){\line(1,0){120}}
\put(225,-45){\line(1,0){120}}
\put(225,-60){\line(1,0){120}}
\put(225,-75){\line(1,0){120}}
\put(225,-90){\line(1,0){120}}
\put(225,-105){\line(1,0){120}}
\put(225,-120){\line(1,0){120}}
\linethickness{0.5mm}
\put(25,0){\line(0,-1){30}}
\put(25,-30){\line(1,0){30}}
\put(55,-30){\line(0,-1){15}}
\put(55,-45){\line(1,0){15}}
\put(70,-45){\line(0,-1){30}}
\put(70,-75){\line(1,0){15}}
\put(85,-75){\line(0,-1){15}}
\put(85,-90){\line(1,0){30}}
\put(115,-90){\line(0,-1){15}}
\put(115,-105){\line(1,0){15}}
\put(130,-105){\line(0,-1){15}}
\put(130,-120){\line(1,0){15}}
\put(225,0){\line(0,-1){30}}
\put(225,-30){\line(1,0){30}}
\put(255,-30){\line(1,0){15}}
\put(270,-30){\line(0,-1){15}}
\put(270,-45){\line(0,-1){30}}
\put(270,-75){\line(1,0){15}}
\put(285,-75){\line(0,-1){15}}
\put(285,-90){\line(1,0){30}}
\put(315,-90){\line(1,0){15}}
\put(330,-90){\line(0,-1){15}}
\put(330,-105){\line(1,0){15}}
\put(345,-105){\line(0,-1){15}}
\end{picture}
\caption{The Hessenberg stair shape and the resulting modified Hessenberg stair shape for $n=8$ and $h=(2,2,3,5,6,6,7,8)$}
\label{modified stair-shap}
\end{figure}
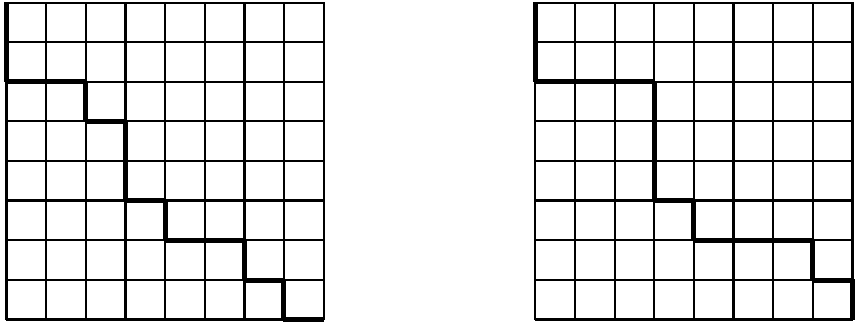

\begin{lemma}\label{Lemma: Negative Hessenberg Root Condition}
If $t_k-t_{\ell}\in\Delta_-$, then 
\[
t_k-t_{\ell} \in \Delta^-_H \quad \text{if and only if} \quad k\leq \overline{h}(\ell).
\]
\end{lemma}
\begin{proof} 
The condition $t_k-t_{\ell} \in \Delta^-_H$ is equivalent to $\mathfrak{g}_{\alpha}\subset H$, where $\alpha=t_k-t_{\ell}$.
Also, the latter condition is equivalent to $k\leq h(\ell)$ via the correspondence \eqref{Hessenberg Correspondence} between Hessenberg functions and Hessenberg subspaces.
Thus it suffices to show that for $k>\ell$, the condition $k\leq h(\ell)$ is equivalent to $k\leq \overline{h}(\ell)$.

Suppose that $k\leq \overline{h}(\ell)$. Since $\overline{h}(\ell)\leq h(\ell)$ (by the definition of $\overline{h}$), we have $k\leq h(\ell)$. 
Conversely, if $k\leq h(\ell)$, then the assumption $k>\ell$ gives $h(\ell)>\ell$. In other words, $h(\ell)\geq \ell+1$, implying that $h(\ell)=\overline{h}(\ell)$. We thus have $k\leq \overline{h}(\ell)$.
\end{proof}

\begin{definition}\label{Definition: Corner}
For $i,j\in\{1,2,\ldots,n\}$, we shall call $(i,j)$ a \textit{corner} of the modified Hessenberg stair shape if $i=\overline{h}(j)$ and $\overline{h}(j-1)<\overline{h}(j)$, with the convention $\overline{h}(0):=0$. 
\end{definition}

\begin{lemma}\label{lemma: corner is off-diagonal}
If $(i,j)$ is a corner of the modified Hessenberg stair shape, then $\overline{h}(j)\neq j$. 
\end{lemma}

\begin{proof}
If $\overline{h}(j)=j$, then in particular $\overline{h}(j)\neq j-1$. The definition \eqref{def of modified function} then implies $h(j)=\overline{h}(j)=j$. Now, \eqref{def of modified function} and $\overline{h}(j)\neq j-1$ also imply that the two conditions $h(j-1)=j-1$ and $h(j)=j$ cannot hold simultaneously. We therefore have $h(j-1)\neq j-1$, which together with $j-1\leq h(j-1)\leq h(j)=j$ implies $h(j-1)=j$. So we have both $\overline{h}(j-1)=j$ and $\overline{h}(j)=j$, and it follows that $(i,j)$ cannot be a corner.
\end{proof}


\begin{lemma}\label{Lemma: contribution from simple roots}
For a simple root $\alpha=t_{j-1}-t_j$ $(2\leq j\leq n)$, we have that
$(j-1,j)$ is a corner of the modified Hessenberg stair shape if and only if $w_{\alpha}\not<w_{\gamma}$ for all $\gamma\in\Delta_H^{-}$.
\end{lemma}

\begin{proof}
To begin, assume that $(j-1,j)$ is a corner. 
By definition, we have $\overline{h}(j)=j-1$. 
Suppose in addition that there exists $\gamma=t_k-t_{\ell}\in\Delta_{H}^{-}$ satisfying $w_{\alpha}< w_{\gamma}$.
Remembering the descriptions of $w_{\alpha}$ and $w_{\gamma}$ from the beginning of this section, we have the following three cases.

(i) $j< \ell$ :
\begin{align*}
w_{\alpha}&=\cdots \ 1 \ n \ \cdots\cdots\cdots\cdots\cdot, \\
w_{\gamma}&=\cdots\cdots\cdots \ n \ \cdots \ 1 \ \cdots.
\end{align*}

(ii) $\ell\leq j$ and $j-1\leq k$ :
\begin{align*}
w_{\alpha}&=\cdots\cdots\cdots \ 1 \ n \ \cdots\cdots\cdot, \\
w_{\gamma}&=\cdots\cdots \ n \ \cdots\cdots \ 1 \ \cdots.
\end{align*}

(iii) $k< j-1$ :
\begin{align*}
w_{\alpha}&=\cdots\cdots\cdots\cdots \ 1 \ n \ \cdots\cdot, \\
w_{\gamma}&=\cdots \ n \ \cdots \ 1 \ \cdots\cdots\cdots.
\end{align*}
Case (i) cannot occur, since transposing $n$ to its right is a length-decreasing process.
Similarly, Case (iii) cannot occur, since transposing $1$ to its left is length-decreasing.
Hence, we must have $\ell\leq j$ and $j-1\leq k$. However, as $t_k-t_{\ell}$ is a negative root, one of these inequalities is strict. Hence, by Lemma \ref{Lemma: Negative Hessenberg Root Condition} 
\begin{align}\label{eq:ineq contr}
j-1\leq k\leq\overline{h}(\ell)\leq\overline{h}(j)=j-1.
\end{align}
It follows that $j-1=k$, so that $\ell<j$ is our strict inequality. However, since $(j-1,j)$ is a corner, we have that $\overline{h}(j-1)<\overline{h}(j)$. Hence, our inequality $\ell<j$ implies that $\overline{h}(\ell)<\overline{h}(j)$.
This contradicts \eqref{eq:ineq contr}, completing the first half of our proof.

We now prove the converse. Suppose that there is no $\gamma\in\Delta_{H}^{-}$ satisfying $w_{\alpha}<w_{\gamma}$. 
We claim that 
\[
 h(j-2)= j-2, \quad 
 h(j-1)= j-1,\quad\text{and}\quad
 h(j)= j,
\]
with the convention $h(0):=0$.
The first of these can be proved as follows. Since the case of $j=2$ is clear, we can assume $j\geq 3$. If $h(j-2)\geq j-1$, then $\gamma:=t_{h(j-2)}-t_{j-2}=t_{\overline{h}(j-2)}-t_{j-2}$ is a negative Hessenberg root by Lemma \ref{Lemma: Negative Hessenberg Root Condition}, and $w_{\alpha}<w_{\gamma}$ since we are in Case (ii). So $h(j-2)= j-2$ follows. The same argument proves $h(j-1)= j-1$ and $h(j)= j$.
Now from the definition of $\overline{h}$, we obtain
\[\overline{h}(j-1)=j-2 \quad \text{and} \quad \overline{h}(j)=j-1.
\]
Hence $(j-1, j)$ is a corner of the modified Hessenberg stair shape.
\end{proof}

Now, recall the definition of $\Omega_H$ from \eqref{Equation: Definition of Omega_H}, as well as the description of the maximal elements of $\Omega_H$ given at the end of \ref{Subsection: Irreducible Components in Type ADE}. With these considerations in mind, Lemma \ref{Lemma: contribution from simple roots} may be restated in the following way: If $\alpha=t_{j-1}-t_j$ is a simple root, then $w_{\alpha}$ is a maximal element of $\Omega_H$ if and only if $(j-1,j)$ is a corner of the modified Hessenberg stair shape. This is consistent with the following more complete description of the maximal elements of $\Omega_H$ in type $A$.
We remind the reader that a corner of the modified Hessenberg stair shape cannot lie on the diagonal (see Lemma \ref{lemma: corner is off-diagonal}).

\begin{proposition}\label{Proposition: Correspondence between corners and maximal elements}
For $\beta=t_i-t_j\in\Delta$,
$w_{\beta}$ is a maximal element of $\Omega_H$ if and only if $(i,j)$ is a corner of the modified Hessenberg stair shape.
\end{proposition}

\begin{proof}
	To prove the backward implication, assume that $(i,j)$ is a corner. We shall distinguish between the cases $\overline{h}(j)=j-1$ and $\overline{h}(j)\neq j-1$. In the former, $(i,j)$ being a corner implies that $i=\overline{h}(j)=j-1$ (so $j-1\geq 1$). In particular, $\beta=t_i-t_j=t_{j-1}-t_j$ is a simple root. Lemma \ref{Lemma: contribution from simple roots} then implies that $w_{\beta}\nless w_{\gamma}$ for all $\gamma\in\Delta_H^{-}$. By the discussion at the end of \ref{Subsection: Irreducible Components in Type ADE}, $w_{\beta}$ is a maximal element of $\Omega_H$.
	
	For our second case, suppose that $\overline{h}(j)\neq j-1$. Since $(i,j)$ is a corner, Lemma \ref{lemma: corner is off-diagonal} implies that $\overline{h}(j)> j$. Again, since $(i,j)$ is a corner, $i=\overline{h}(j)$. In particular, $i>j$ and $t_{i}-t_{j}$ is a negative Hessenberg root. As $(i,j)$ is a corner with $i>j$, an application of Lemma \ref{Lemma: Negative Hessenberg Root Condition} establishes that $t_i-t_j$ is a minimal element of $\Delta_H^{-}$. The discussion at the end of \ref{Subsection: Irreducible Components in Type ADE} then shows that $w_{\beta}$ is a maximal element of $\Omega_H$.
	
	We now prove the forward implication. Firstly, assume that $\beta=t_i-t_j$ is simple (so $i=j-1$). By Lemma \ref{Lemma: contribution from simple roots}, $(i,j)=(j-1,j)$ is a corner of the modified Hessenberg stair shape.
	
	Secondly, assume that $\beta=t_i-t_j$ is a minimal element of $\Delta_{H}^-$ (so $i> j\geq 1$). We have $\overline{h}(j)=i$, since $t_{\overline{h}(j)}-t_j$ would otherwise be a strictly less than $t_i-t_j$. A similar argument establishes that $\overline{h}(j-1) <\overline{h}(j)$ must also hold, so that $(i,j)$ is a corner.  
\end{proof}

As noted earlier, the irreducible components of $X_H(e_{\theta})$ correspond to the maximal elements of $\Omega_H$. Noting that these maximal elements are described in Proposition \ref{Proposition: Correspondence between corners and maximal elements}, the following theorem gives the irreducible components of $X_H(e_{\theta})$ in Lie type $A_{n-1}$

\begin{theorem}\label{Theorem: Bijective Correspondence in Type A}
In type $A_{n-1}$, there is a bijective correspondence between the set of corners of the modified Hessenberg stair shape and the set of irreducible components of $X_H(e_{\theta})$ given by
\[
(\overline{h}(j), j) \mapsto X(w_j)=\overline{Bw_j B/B},
\]
where $w_j$ is the longest permutation satisfying $w_j(\overline{h}(j))=1$ and $w_j(j)=n$.
\end{theorem}

Let us implement Theorem \ref{Theorem: Bijective Correspondence in Type A} in the context of a specific example. Indeed, recall that Figure \ref{modified stair-shap} includes the modified Hessenberg stair shape determined by $h=(2,2,3,5,6,6,7,8)$ when $n=8$. The corners are $(2,1)$, $(5,4)$, $(6,5)$, and $(7,8)$, as is indicated in the following diagram.

\begin{figure}[h]
	\begin{picture}(350,130)(120,-130)
	\put(225,0){\line(0,-1){120}}
	\put(240,0){\line(0,-1){120}}
	\put(255,0){\line(0,-1){120}}
	\put(270,0){\line(0,-1){120}}
	\put(285,0){\line(0,-1){120}}
	\put(300,0){\line(0,-1){120}}
	\put(315,0){\line(0,-1){120}}
	\put(330,0){\line(0,-1){120}}
	\put(345,0){\line(0,-1){120}}
	\put(225,0){\line(1,0){120}}
	\put(225,-15){\line(1,0){120}}
	\put(225,-30){\line(1,0){120}}
	\put(225,-45){\line(1,0){120}}
	\put(225,-60){\line(1,0){120}}
	\put(225,-75){\line(1,0){120}}
	\put(225,-90){\line(1,0){120}}
	\put(225,-105){\line(1,0){120}}
	\put(225,-120){\line(1,0){120}}
	\linethickness{0.5mm}
	\put(225,0){\line(0,-1){30}}
	\put(225,-30){\line(1,0){30}}
	\put(255,-30){\line(1,0){15}}
	\put(270,-30){\line(0,-1){15}}
	\put(270,-45){\line(0,-1){30}}
	\put(270,-75){\line(1,0){15}}
	\put(285,-75){\line(0,-1){15}}
	\put(285,-90){\line(1,0){30}}
	\put(315,-90){\line(1,0){15}}
	\put(330,-90){\line(0,-1){15}}
	\put(330,-105){\line(1,0){15}}
	\put(345,-105){\line(0,-1){15}}
	\put(232.5,-22.5){\circle*{5}}
	\put(277.5,-67.5){\circle*{5}}
	\put(292.5,-82.5){\circle*{5}}
	\put(337.5,-97.5){\circle*{5}}
	\end{picture}
	\caption{The modified Hessenberg stair shape for $h=(2,2,3,5,6,6,7,8)$ with dots labeling corners}
	\label{corners of modified stair-shap}
\end{figure}
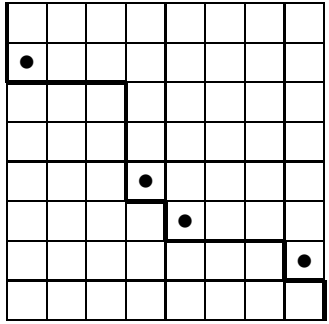

By Theorem \ref{Theorem: Bijective Correspondence in Type A}, the irreducible components of $X_H(e_{\theta})$ are the Schubert varieties $X(w)$ for the following elements $w\in S_8$:  
\[
8\ 1\ 7\ 6\ 5\ 4\ 3\ 2, \quad 
7\ 6\ 5\ 8\ 1\ 4\ 3\ 2, \quad
7\ 6\ 5\ 4\ 8\ 1\ 3\ 2, \quad
7\ 6\ 5\ 4\ 3\ 2\ 1\ 8.
\]

\section{GKM Theory on $X_H(e_{\theta})$}\label{Section: GKM Theory on Minimal Nilpotent Hessenberg Varieties}
We devote this section to the construction and examination of a GKM variety structure (see \cite{Goresky}) on $X_H(e_{\theta})$. Let us begin by reviewing the relevant parts of GKM theory. 
\subsection{Brief Review of GKM Theory}\label{Section: Brief Review of GKM Theory} Let $X$ be a complex projective variety acted upon algebraically by $T$, where $T$ is the maximal torus fixed in \ref{Section: Notation}. One calls $X$ a \textit{GKM variety} when the following conditions are satisfied.
\begin{itemize}
\item[(i)] $X^T$ is finite.
\item[(ii)] $X$ has finitely many one-dimensional $T$-orbits.
\item[(iii)] If $Y\subseteq X$ is a one-dimensional $T$-orbit, then $\overline{Y}$ is $T$-equivariantly isomorphic to $\mathbb{P}^1$ with the $T$-action $t\cdot[x_1:x_2]=[\alpha(t)x_1:x_2]$ for some non-zero weight $\alpha\in X^*(T)$.\footnote{Note that this weight is only determined up to sign.}
\item[(iv)] $X$ is $T$-equivariantly formal, meaning that the spectral sequence of the natural fibration $X\rightarrow X_T\rightarrow BT$ collapses on its second page. 
\end{itemize}

Let us assume this to be the case. Now, write $X^T=\{x_1,x_2,\ldots,x_n\}$ with $x_i\neq x_j$ for $i\neq j$. Suppose that $i\neq j$ and that $x_i,x_j\in \overline{Y_{ij}}$ for some one-dimensional $T$-orbit $Y_{ij}\subseteq X$. In this case, we shall write $i\leftrightarrow j$. Note that $\overline{Y_{ij}}$ is acted upon by $T$ with some non-zero weight $\alpha_{ij}\in X^*(T)$, as in (iii).

The restriction map $$H_T^*(X)\rightarrow H_T^*(X^T)=\bigoplus_{i=1}^nH_T^*(\{x_i\})=\bigoplus_{i=1}^nH_T^*(\text{pt})$$ is injective, and its image is precisely \begin{equation}\label{Image Description}H_T^*(X)\cong\{(f_1,f_2,\ldots,f_n)\in\bigoplus_{i=1}^nH_T^*(\text{pt}):\alpha_{ij}\vert(f_i-f_j)\text{ whenever }i\leftrightarrow j\}.\end{equation} The image description \eqref{Image Description} is naturally encoded in an edge-labelled graph, called the \textit{GKM graph} of $X$. This graph has vertex set $\{1,2,\ldots,n\}$, with $i$ and $j$ connected by an edge if and only if $i\leftrightarrow j$ as defined above. In this case, the edge in question is given the label $\alpha_{ij}$.\footnote{As mentioned earlier, this edge label is well-defined only up to sign. However, the image \eqref{Image Description} is clearly unaffected by the choice of sign.}

At a later stage, it will be convenient to have the following definition at our disposal.

\begin{definition}\label{GKM Subvariety}
A closed subvariety $Z\subseteq X$ is called a \textit{GKM subvariety} if $Z$ is $T$-invariant and is itself a GKM variety with respect to the $T$-action. Equivalently, $Z$ is a GKM subvariety if $Z$ is $T$-invariant and $T$-equivariantly formal.
\end{definition}

We note that if $Z\subseteq X$ is a GKM subvariety, the GKM graph of $Z$ is canonically a sub-(labelled) graph of the GKM graph of $X$.   

\subsection{The GKM Graph of $G/B$}\label{Section: The GKM Graph of G/B}
It will be advantageous to briefly review the GKM variety structure on $G/B$ with its usual $T$-action. Having presented $(G/B)^T$ in \eqref{T-Fixed Points in G/B}, it just remains to describe the one-dimensional $T$-orbits in $G/B$. Given $\alpha\in\Delta_{+}$, denote by $SL_2(\mathbb{C})_{\alpha}\subseteq G$ the root subgroup with Lie algebra
$\mathfrak{g}_{-\alpha}\oplus[\mathfrak{g}_{-\alpha},\mathfrak{g}_{\alpha}]\oplus\mathfrak{g}_{\alpha}\subseteq\mathfrak{g}$. We define $Y_{e,\alpha}$ to be the $\SL_2(\mathbb{C})_{\alpha}$-orbit of $x_e$ in $G/B$, where $e\in W$ is the identity element. For an arbitrary element $w\in W$, write $w=[g]$ for $g\in N_G(T)$. We define $Y_{w,\alpha}$ to be the left $g$-translate of $Y_{e,\alpha}$, namely $$Y_{w,\alpha}:=gY_{e,\alpha}.$$ This $T$-invariant closed subvariety of $G/B$ is isomorphic to $\mathbb{P}^1$. Also, $$(Y_{w,\alpha})^T=\{x_w,x_{ws_{\alpha}}\},$$ while $Y_{w,\alpha}\setminus\{x_w,x_{ws_{\alpha}}\}$ is a one-dimensional $T$-orbit. In fact, it is known that every one-dimensional $T$-orbit in $G/B$ is of this form.

In light of the above, the GKM graph of $G/B$ has vertex set $W$, and there is an edge connecting $w,w'\in W$ if and only if $w'=ws_{\alpha}$ for some $\alpha\in\Delta_{+}$. The edge connecting $w$ and $ws_{\alpha}$ is then seen to be labelled with the weight $w\alpha$. In other words, the image of $H_T^*(G/B)\rightarrow H_T^*((G/B)^T)=\bigoplus_{w\in W}H_T^*(\text{pt})$ is \begin{equation}\label{Image of Localization Map for G/B}\{(f_w)\in\bigoplus_{w\in W}H_T^*(\text{pt}):(w\alpha)\vert(f_w-f_{ws_{\alpha}})\text{ }\forall w\in W,\text{ }\alpha\in\Delta_{+}\}.\end{equation} 

In the interest of examples to be considered later, let us construct the GKM graph of the flag variety of $G=\SL_3(\mathbb{C})$. Let $T\subseteq SL_3(\mathbb{C})$ and $B\subseteq\SL_3(\mathbb{C})$ be the maximal torus of diagonal matrices and the Borel of upper-triangular matrices, respectively. The positive roots are given by $$\alpha_1:= t_1-t_2,$$ $$\alpha_2:=t_2-t_3,$$ and $$\alpha_3:= t_1-t_3.$$ Note that $W=S_3$, and that in one-line notation, $s_{\alpha_1}=2 \ 1 \ 3$, $s_{\alpha_2}=1 \ 3 \ 2$, and $s_{\alpha_3}=3 \ 2 \ 1$. Hence, the GKM graph of $\SL_3(\mathbb{C})/B\cong Flags(\mathbb{C}^3)$ is as follows.

\begin{figure}[h]
\begin{picture}(350,120)(0,-60)
\put(175,45){\circle*{5}}
\put(175,-45){\circle*{5}}
\put(125,20){\circle*{5}}
\put(225,20){\circle*{5}}
\put(125,-20){\circle*{5}}
\put(225,-20){\circle*{5}}
\put(175,45){\line(-2,-1){50}}
\put(175,45){\line(2,-1){50}}
\put(175,45){\line(0,-1){90}}
\put(125,20){\line(0,-1){40}}
\put(125,20){\line(5,-2){100}}
\put(225,20){\line(0,-1){40}}
\put(225,20){\line(-5,-2){100}}
\put(125,-20){\line(2,-1){50}}
\put(225,-20){\line(-2,-1){50}}
\put(165,-60){$1 \ 2 \ 3$}
\put(165,52){$3 \ 2 \ 1$}
\put(100,25){$2 \ 3 \ 1$}
\put(225,-35){$1 \ 3 \ 2$}
\put(225,25){$3 \ 1 \ 2$}
\put(100,-35){$2 \ 1 \ 3$}
\put(140,-40){$\alpha_1$}
\put(195,-40){$\alpha_2$}
\put(110,0){$\alpha_3$}
\put(230,0){$\alpha_3$}
\put(140,37){$\alpha_2$}
\put(195,37){$\alpha_1$}
\put(140,-7){$\alpha_2$}
\put(195,-7){$\alpha_1$}
\put(177,20){$\alpha_3$}
\end{picture}
\caption{The GKM graph of $\SL_3(\mathbb{C})/B\cong Flags(\mathbb{C}^3)$}
\label{GKM Graph of Flags(C^3)}
\end{figure}
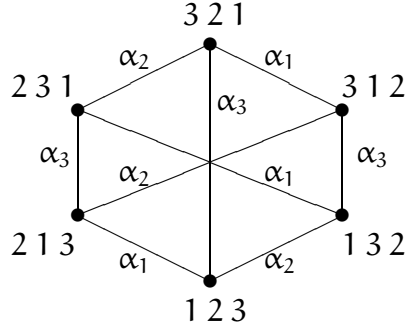

\subsection{The GKM Graph of $X_H(e_{\theta})$}\label{Section: The GKM Graph}
Since Proposition \ref{B-invariance} shows $X_H(e_{\theta})$ to be a union of Schubert cells, this variety has trivial cohomology in odd grading degrees. It follows that $X_H(e_{\theta})$ is $T$-equivariantly formal (see \cite{Goresky}, Section 14), and hence a GKM subvariety of $G/B$. Accordingly, we will describe $H_T^*(X_H(e_{\theta}))$ by exhibiting the GKM graph of $X_H(e_{\theta})$ as a subgraph of the GKM graph of $G/B$. Noting that the vertices of our subgraph have been determined by Proposition \ref{Fixed Points}, we need only determine the edges.  For this latter part, we will need to briefly discuss root strings. 

If $\alpha,\beta\in\Delta$ are roots, one has the root string $$S(\beta,\alpha):=(\Delta\cup\{0\})\cap\{\beta+n\alpha:n\in\mathbb{Z}\}.$$ If $p,q\in\mathbb{Z}$ are maximal for the properties $\beta+p\alpha\in S(\beta,\alpha)$ and $\beta-q\alpha\in S(\beta,\alpha)$, respectively, then $$S(\beta,\alpha)=\{\beta+n\alpha:-q\leq n\leq p\}$$ and \begin{equation}\label{Root String}q-p=\frac{2(\beta,\alpha)}{(\alpha,\alpha)}\end{equation} (see Proposition 2.29 of \cite{Knapp}).
The relevance of root strings to our present work is captured by the following lemma.
\begin{lemma}\label{Lowest Root}
If $w\in W$ and $\alpha\in\Delta_{+}$ are such that $x_{ws_{\alpha}}\in X_H(e_{\theta})^T$, then $$\bigoplus_{\beta\in S(w^{-1}\theta,\alpha)}\mathfrak{g}_{\beta}\subseteq H.\footnote{Here, it is understood that $\mathfrak{g}_0=[\mathfrak{g}_{-\alpha},\mathfrak{g}_{\alpha}]$.}$$
\end{lemma}

\begin{proof}
First note that either $\alpha\in w^{-1}\Delta_{+}$ or $-\alpha\in w^{-1}\Delta_{+}$. Since that $S(w^{-1}\theta,\alpha)=S(w^{-1}\theta,-\alpha)$, we may assume that $\alpha\in w^{-1}\Delta_{+}$ (ie. that $w\alpha\in\Delta_{+}$). Noting that $\theta$ is the highest root, \eqref{Root String} implies that $$S(\theta,w\alpha)=\{\theta-n(w\alpha),\theta-(n-1)(w\alpha),\ldots,\theta-w\alpha,\theta\}$$ for $n=\frac{2(\theta,w\alpha)}{(w\alpha,w\alpha)}=\frac{2(w^{-1}\theta,\alpha)}{(\alpha,\alpha)}.$ Acting on the above string by $w^{-1}$, we see that $S(w^{-1}\theta,\alpha)$ is given by $$S(w^{-1}\theta,\alpha)=\{w^{-1}\theta-n\alpha,w^{-1}\theta-(n-1)\alpha,\ldots,w^{-1}\theta-\alpha,w^{-1}\theta\}.$$ The lowest root in this string is $w^{-1}\theta-n\alpha=s_{\alpha}(w^{-1}\theta)$. Also, applying Proposition \ref{Fixed Points} to the condition $x_{ws_{\alpha}}\in X_H(e_{\theta})^T$ gives $$\mathfrak{g}_{w^{-1}\theta-n\alpha}=\mathfrak{g}_{s_{\alpha}(w^{-1}\theta)}\subseteq H.$$ Since $H$ is $\mathfrak{b}$-invariant, repeated bracketing with $\mathfrak{g}_{\alpha}\subseteq\mathfrak{b}$ establishes that the root space of each root in $S(w^{-1}\theta,\alpha)$ lies in $H$. This completes the proof. \end{proof}

\begin{theorem}\label{GKM Graph Theorem}
The GKM graph of $X_H(e_{\theta})$ is a full subgraph of the GKM graph of $G/B$.
\end{theorem}

\begin{proof}
Equivalently, we claim that if $w\in W$ and $\alpha\in\Delta_{+}$ are such that $x_{w},x_{ws_{\alpha}}\in X_H(e_{\theta})^T$, then $Y_{w,\alpha}\subseteq X_H(e_{\theta})$. To this end, fix a representative $g\in N_G(T)$ of $w$, and let $N_{-\alpha}$ denote the connected closed subgroup of $\SL_2(\mathbb{C})_{\alpha}$ with Lie algebra $\mathfrak{g}_{-\alpha}$. Note that $$Y_{e,\alpha}=\overline{N_{-\alpha}x_e},$$ the closure of the $N_{-\alpha}$-orbit through $x_e$. We therefore have $$Y_{w,\alpha}=gY_{e,\alpha}=\overline{(gN_{-\alpha})x_e}.$$ Since $X_H(e_{\theta})$ is a closed subvariety of $G/B$, proving that $(gN_{-\alpha})x_e\subseteq X_H(e_{\theta})$ will establish that $Y_{w,\alpha}\subseteq X_H(e_{\theta})$. To prove the former, it will suffice to establish that $gh\in G_H(e_{\theta})$ for all $h\in N_{-\alpha}$, namely $\Adj_{(gh)^{-1}}(e_{\theta})\in H$. 

Suppose that $h\in N_{-\alpha}$. Writing $\Adj_{g^{-1}}(e_{\theta})=e_{w^{-1}\theta}\in\mathfrak{g}_{w^{-1}\theta}$ and $h=\exp(\xi)$ for $\xi\in\mathfrak{g}_{-\alpha}$, we obtain  \begin{equation}\label{Computation}\Adj_{(gh)^{-1}}(e_{\theta})=\Adj_{\exp(-\xi)}(e_{w^{-1}\theta})=e^{\adj_{-\xi}}(e_{w^{-1}\theta})=\sum_{k=0}^{\infty}\frac{1}{k!}(\adj_{-\xi})^k(e_{w^{-1}\theta})\end{equation} Furthermore, if $(\adj_{-\xi})^k(e_{w^{-1}\theta})\neq 0$, then it belongs to a root space for a root in $S(w^{-1}\theta,\alpha)$. Hence, \eqref{Computation} implies that $$\Adj_{(gh)^{-1}}(e_{\theta})\in\bigoplus_{\beta\in S(w^{-1}\theta,\alpha)}\mathfrak{g}_{\beta}.$$ By Lemma \ref{Lowest Root}, it follows that $\Adj_{(gh)^{-1}}(e_{\theta})\in H$.
\end{proof}

\begin{remark}
Using Carrell's results \cite{Carrell}, which show the GKM graph of a Schubert variety to be a full subgraph of the GKM graph of $G/B$, one can write an alternative proof of Theorem \ref{GKM Graph Theorem}.
\end{remark}

Combining Proposition \ref{Fixed Points} and Theorem \ref{GKM Graph Theorem}, one finds the image of the restriction map $$H_T^*(X_H(e_{\theta}))\rightarrow H_T^*(X_H(e_{\theta})^T)=\bigoplus_{\substack{w\in W \\ w^{-1}\theta\in\Delta_H}}H_T^*(\text{pt})$$ to be 
\begin{align*}
H^*_T(X_H(e_{\theta})) \cong \left\{ (f_w) \in \bigoplus_{\substack{w\in W \\ w^{-1}\theta\in\Delta_H}} H_T^*(\text{pt})\left\vert
\begin{matrix} \text{ $f_{w}-f_{ws_{\alpha}}$ is divisible by $w\alpha$} \\
\text{ for $\alpha\in\Delta_+$ and $w\in W$ satisfying } \\
 \text{$w^{-1}\theta\in\Delta_H$ and $(ws_{\alpha})^{-1}\theta\in\Delta_H$} \end{matrix} \right.\right\}.
\end{align*}

\subsection{GKM Graphs of $X_H(e_{\theta})$ in Type $A_2$}\label{Section: GKM Graphs in Type A}
By Theorem \ref{GKM Graph Theorem}, finding the GKM graph of $X_H(e_{\theta})$ amounts to determining its $T$-fixed points. With this in mind, suppose that $G=\SL_3(\mathbb{C})$ and that $T\subseteq\SL_3(\mathbb{C})$ and $B\subseteq\SL_3(\mathbb{C})$ are the maximal torus and Borel considered in \ref{Section: Examples in Type A}, respectively. Recall that $W=S_3$ and let $h:\{1,2,3\}\rightarrow\{1,2,3\}$ be a Hessenberg function corresponding to $H\subseteq\mathfrak{sl}_3(\mathbb{C})$. 

The possible Hessenberg functions are $(1,2,3)$, $(1,3,3)$, $(2,2,3)$, $(2,3,3)$, and $(3,3,3)$. By applying \eqref{Type A Fixed Points}, one determines the $T$-fixed points for each corresponding variety $X_H(e_{\theta})$. Noting that Theorem \ref{GKM Graph Theorem} then determines the GKM graph of each variety as a subgraph of Figure \ref{GKM Graph of Flags(C^3)}, the following are the GKM graphs of all the $X_H(e_{\theta})$ in type $A_2$. 

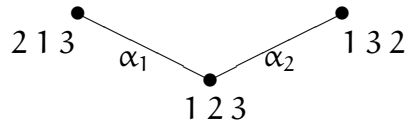
\begin{figure}[h]
\begin{picture}(350,60)(0,-60)
\put(175,-45){\circle*{5}}
\put(125,-20){\circle*{5}}
\put(225,-20){\circle*{5}}
\put(125,-20){\line(2,-1){50}}
\put(225,-20){\line(-2,-1){50}}
\put(165,-60){$1 \ 2 \ 3$}
\put(225,-35){$1 \ 3 \ 2$}
\put(100,-35){$2 \ 1 \ 3$}
\put(140,-40){$\alpha_1$}
\put(195,-40){$\alpha_2$}
\end{picture}
\caption{The GKM graph of $X_H(e_{\theta})$ for $h=(1,2,3)$}
\end{figure}

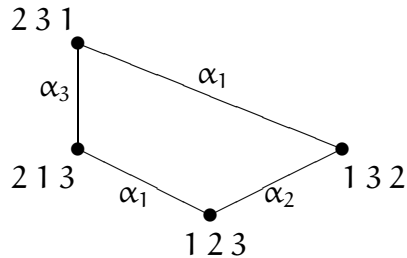
\begin{figure}[h]
\begin{picture}(350,90)(0,-60)
\put(175,-45){\circle*{5}}
\put(125,20){\circle*{5}}
\put(125,-20){\circle*{5}}
\put(225,-20){\circle*{5}}
\put(125,20){\line(0,-1){40}}
\put(125,20){\line(5,-2){100}}
\put(125,-20){\line(2,-1){50}}
\put(225,-20){\line(-2,-1){50}}
\put(165,-60){$1 \ 2 \ 3$}
\put(100,25){$2 \ 3 \ 1$}
\put(225,-35){$1 \ 3 \ 2$}
\put(100,-35){$2 \ 1 \ 3$}
\put(140,-40){$\alpha_1$}
\put(195,-40){$\alpha_2$}
\put(110,0){$\alpha_3$}
\put(170,5){$\alpha_1$}
\end{picture}
\caption{The GKM graph of $X_H(e_{\theta})$ for $h=(1,3,3)$}
\end{figure}

\begin{figure}[h]
\begin{picture}(350,90)(0,-60)
\put(175,-45){\circle*{5}}
\put(225,20){\circle*{5}}
\put(125,-20){\circle*{5}}
\put(225,-20){\circle*{5}}
\put(225,20){\line(0,-1){40}}
\put(225,20){\line(-5,-2){100}}
\put(125,-20){\line(2,-1){50}}
\put(225,-20){\line(-2,-1){50}}
\put(165,-60){$1 \ 2 \ 3$}
\put(225,-35){$1 \ 3 \ 2$}
\put(225,25){$3 \ 1 \ 2$}
\put(100,-35){$2 \ 1 \ 3$}
\put(140,-40){$\alpha_1$}
\put(195,-40){$\alpha_2$}
\put(230,0){$\alpha_3$}
\put(165,5){$\alpha_2$}
\end{picture}
\caption{The GKM graph of $X_H(e_{\theta})$ for $h=(2,2,3)$}
\end{figure}
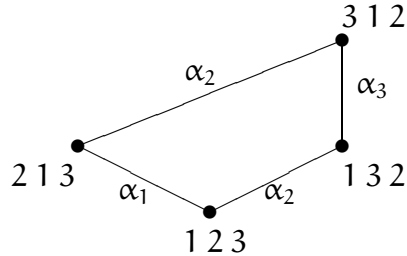

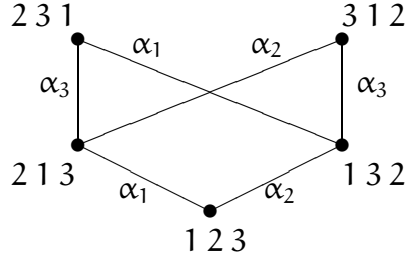
\begin{figure}[h]
\begin{picture}(350,90)(0,-60)
\put(175,-45){\circle*{5}}
\put(125,20){\circle*{5}}
\put(225,20){\circle*{5}}
\put(125,-20){\circle*{5}}
\put(225,-20){\circle*{5}}
\put(125,20){\line(0,-1){40}}
\put(125,20){\line(5,-2){100}}
\put(225,20){\line(0,-1){40}}
\put(225,20){\line(-5,-2){100}}
\put(125,-20){\line(2,-1){50}}
\put(225,-20){\line(-2,-1){50}}
\put(165,-60){$1 \ 2 \ 3$}
\put(100,25){$2 \ 3 \ 1$}
\put(225,-35){$1 \ 3 \ 2$}
\put(225,25){$3 \ 1 \ 2$}
\put(100,-35){$2 \ 1 \ 3$}
\put(140,-40){$\alpha_1$}
\put(195,-40){$\alpha_2$}
\put(110,0){$\alpha_3$}
\put(230,0){$\alpha_3$}
\put(145,15){$\alpha_1$}
\put(190,15){$\alpha_2$}
\end{picture}
\caption{The GKM graph of $X_H(e_{\theta})$ for $h=(2,3,3)$}
\label{fig:GKM Graph for h=(2,3,3)}
\end{figure}

\begin{figure}[h]
\begin{picture}(350,120)(0,-60)
\put(175,45){\circle*{5}}
\put(175,-45){\circle*{5}}
\put(125,20){\circle*{5}}
\put(225,20){\circle*{5}}
\put(125,-20){\circle*{5}}
\put(225,-20){\circle*{5}}
\put(175,45){\line(-2,-1){50}}
\put(175,45){\line(2,-1){50}}
\put(175,45){\line(0,-1){90}}
\put(125,20){\line(0,-1){40}}
\put(125,20){\line(5,-2){100}}
\put(225,20){\line(0,-1){40}}
\put(225,20){\line(-5,-2){100}}
\put(125,-20){\line(2,-1){50}}
\put(225,-20){\line(-2,-1){50}}
\put(165,-60){$1 \ 2 \ 3$}
\put(165,52){$3 \ 2 \ 1$}
\put(100,25){$2 \ 3 \ 1$}
\put(225,-35){$1 \ 3 \ 2$}
\put(225,25){$3 \ 1 \ 2$}
\put(100,-35){$2 \ 1 \ 3$}
\put(140,-40){$\alpha_1$}
\put(195,-40){$\alpha_2$}
\put(110,0){$\alpha_3$}
\put(230,0){$\alpha_3$}
\put(140,37){$\alpha_2$}
\put(195,37){$\alpha_1$}
\put(140,-7){$\alpha_2$}
\put(195,-7){$\alpha_1$}
\put(177,20){$\alpha_3$}
\end{picture}
\caption{The GKM graph of $X_H(e_{\theta})$ for $h=(3,3,3)$}
\end{figure}
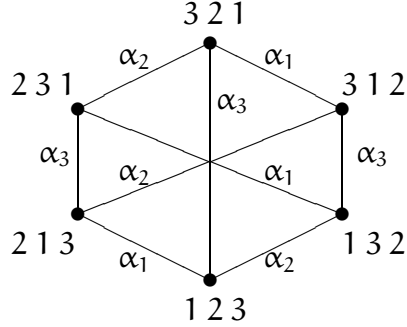

\section{Cohomology Ring Presentations}\label{Section: Cohomology Ring Presentations}
In this section, we use the restriction maps $$i^*:H^*(G/B)\rightarrow H^*(X_H(e_{\theta}))\text{ and }i_T^*:H_T^*(G/B)\rightarrow H_T^*(X_H(e_{\theta}))$$ to explicitly present $H^*(X_H(e_{\theta}))$ and $H_T^*(X_H(e_{\theta}))$ as quotients of $H^*(G/B)$ and $H_T^*(G/B)$, respectively.

\subsection{Ordinary Cohomology}\label{Section: Ordinary Cohomology}
We begin with the following proposition. 

\begin{proposition}\label{Surjectivity on Ordinary Cohomology}
The restriction map $i^*:H^*(G/B)\rightarrow H^*(X_H(e_{\theta}))$ is surjective.
\end{proposition}

\begin{proof}
Since rational singular cohomology is the dual of rational singular homology, it suffices to show that the map $H_*(X_H(e_{\theta}))\rightarrow H_*(G/B)$ is injective. 
Now, consider the commutative diagram 

\begin{equation}\label{Commutative Diagram for Borel-Moore}
\begin{CD}
H_*(X_H(e_{\theta}))@>{\cong}>> \overline{H}_*(X_H(e_{\theta}))\\
@V{ }VV @V{ }VV\\
H_*(G/B) @>{\cong}>> \overline{H}_*(G/B)
\end{CD},
\end{equation}
where $\overline{H}_*$ denotes Borel-Moore homology (see \cite{Fulton}); for a closed embedding of a topological space $Y$ into some Euclidean space $\mathbb{R}^m$, the Borel-Moore homology $\overline{H}_i(Y)$, which is defined up to isomorphism, is given by $H^{m-i}(\mathbb{R}^m,\mathbb{R}^m\setminus Y)$. 
The vertical maps in \eqref{Commutative Diagram for Borel-Moore} are the maps induced by the inclusion $X_H(e_{\theta})\hookrightarrow G/B$, and the horizontal isomorphisms are the ones described in 6.10.14 of \cite{Spanier}.
So what we need check is that the induced map $\overline{H}_*(X_H(e_{\theta}))\rightarrow \overline{H}_*(G/B)$ is injective.
To this end, consider the subsets 

$$(G/B)_p := \coprod_{w\in W, \ \ell(w)\leq p}BwB/B$$
and
$$(X_H(e_{\theta}))_p := \coprod_{\substack{w\in W, \ \ell(w)\leq p \\ x_w\in X_H(e_{\theta})^T}}BwB/B$$
for $p\in\mathbb{Z}_{\geq 0}$. Since Proposition \ref{B-invariance} tells us that $X_H(e_{\theta})$ is a union of Schubert cells, we have
the affine pavings
$$G/B=(G/B)_{\dim_{\mathbb{C}}(G/B)}\supseteq \ldots \supseteq (G/B)_1 \supseteq (G/B)_0=\emptyset$$
and 
$$X_H(e_{\theta})=(X_H(e_{\theta}))_{\dim_{\mathbb{C}}(X_H(e_{\theta}))}\supseteq \ldots \supseteq (X_H(e_{\theta}))_1 \supseteq (X_H(e_{\theta}))_0=\emptyset.$$
Also, for each $p=0,1,\ldots, \dim_{\mathbb{C}}(G/B)$, we have a commutative diagram (c.f. \cite{Fulton})

\begin{equation*}
\begin{CD}
0@>{}>>\overline{H}_*((X_H(e_{\theta}))_{p-1})@>{ }>> \overline{H}_*((X_H(e_{\theta}))_p)@>{ }>> \displaystyle{\bigoplus_{\substack{\ell(w)=p \\ x_w\in X_H(e_{\theta})^T}}}\overline{H}_*(BwB/B)@>{}>>0\\
@. @V{ }VV @V{ }VV @V{ }VV\\
0@>{}>>\overline{H}_*((G/B)_{p-1})@>{ }>> \overline{H}_*((G/B)_p) @>{ }>>\displaystyle{\bigoplus_{\ell(w)=p}}\overline{H}_*(BwB/B)@>{}>>0.
\end{CD}
\end{equation*}

The left and the middle vertical maps are those induced by the inclusions, and each component of the right vertical map is the composition $\overline{H}_*(BwB/B)\rightarrow \overline{H}_*(\coprod_{\ell(w)=p}BwB/B)\rightarrow \overline{H}_*(BwB/B)$. It is straightforward to see that each component map is an isomorphim.
Hence, the right vertical map is an injection, and we see that the induced map $\overline{H}_*(X_H(e_{\theta}))\rightarrow \overline{H}_*(G/B)$ is injective by induction on $p$. This completes the proof.
\end{proof}

In light of Proposition \ref{Surjectivity on Ordinary Cohomology}, we shall address ourselves to computing the kernel of $i^*$. To this end, we will need the following proposition. 

\begin{proposition}\label{Schubert Variety}
Suppose that $v,w\in W$ satisfy $v\geq w$. If $x_w\not\in X_H(e_{\theta})^T$, then $x_v\not\in X_H(e_{\theta})^T$.
\end{proposition}

\begin{proof}
Suppose that $x_v\in X_H(e_{\theta})^T$. Since $X_H(e_{\theta})$ is a closed $B$-invariant subvariety of $G/B$, it follows that $\overline{BvB/B}\subseteq X_H(e_{\theta})$. Noting that $v\geq w$, we must have $x_w\in\overline{BvB/B}$. This contradicts our assumption that $x_w\not\in X_H(e_{\theta})^T$.
\end{proof}

Now, recall that $B_{-}\subseteq G$ denotes our opposite Borel subgroup. One then has the opposite Schubert varieties
$$X_{-}(w):=\overline{B_{-}wB/B},\text{ }\text{ }w\in W.$$ Note that $X_{-}(w)$ determines an opposite Schubert class $\sigma(w)\in H^{2\ell(w)}(G/B)$. 

\begin{corollary}\label{corollary for ideal}
The linear subspace \begin{equation}\label{Ideal Definition} J_H:=\bigoplus_{x_w\not\in X_H(e_{\theta})^T}\mathbb{Q}\sigma(w)\end{equation} is an ideal of $H^*(G/B)$.
\end{corollary}

\begin{proof}
If $u\in W$ and $x_w\notin X_H(e)^T$, then ordinary Schubert calculus gives \begin{equation}\label{Schubert Calculus}\sigma(u)\sigma(w)=\sum_{v\geq u,w}c_{uw}^v\sigma(v)\end{equation} for some $c_{uw}^v\in\mathbb{Z}$. By Proposition \ref{Schubert Variety}, $x_v\not\in X_H(e_{\theta})^T$ for all $v$ appearing in the sum \eqref{Schubert Calculus}. Hence, $\sigma(u)\sigma(w)\in J_H$, proving that $J_H$ is an ideal.
\end{proof}

With these considerations in mind, we offer the following presentation of $H^*(X_H(e_{\theta}))$.
\begin{theorem}\label{Ordinary Kernel}
The map $i^*$ induces a graded $\mathbb{Q}$-algebra isomorphism
\begin{align*}
H^*(G/B)/J_H \rightarrow H^*(X_H(e_{\theta})).
\end{align*}
\end{theorem}
\proof
To begin, we claim that $i^*(\sigma(w))=0$ for $w\in W$ satisfying $x_w\not\in X_H(e_{\theta})^T$.  
This will follow from our establishing that
\begin{align}\label{Empty Intersection}
X_{-}(w)\cap X_H(e_{\theta})=\emptyset.
\end{align}
To this end, we have 
\begin{align}\label{Decompositions}
X_{-}(w)=\coprod_{w\leq v}B_-vB/B 
\quad \text{and} \quad
X_H(e_{\theta})=\coprod_{x_u\in X_H(e_{\theta})^T}BuB/B,
\end{align}
with the latter decomposition being a consequence of Proposition \ref{B-invariance}. Now, recall that for $u,v\in W$, $B_-vB/B\cap BuB/B\neq\emptyset$ if and only if $v\leq u$ (see \cite{Deodhar}, Corollary 1.2). So, if $X_{-}(w)\cap X_H(e_{\theta})\neq\emptyset$, then \ref{Decompositions} implies $w\leq u$ for some $u\in W$ with $x_u\in X_H(e_{\theta})^T$. Proposition \ref{Schubert Variety} then gives $x_w\in X_H(e_{\theta})^T$, which is a contradiction. We conclude that \eqref{Empty Intersection} holds, so that $i^*(\sigma(w))=0$ whenever $x_w\not\in X_H(e_{\theta})^T$.

In light of our findings, $i^*$ induces a surjective graded $\mathbb{Q}$-algebra homomorphism
\begin{align}\label{Isomorphism}
H^*(G/B)/J_H \rightarrow H^*(X_H(e_{\theta})).
\end{align}
To conclude that \eqref{Isomorphism} is an isomorphism, it will suffice to prove that \begin{equation}\label{Dimensions}
\dim_{\mathbb{Q}}(H^*(G/B)/J_{H})=\dim_{\mathbb{Q}}(H^*(X_H(e_{\theta}))).
\end{equation} 
Noting that
\begin{align*}
H^*(G/B) = \bigoplus_{w\in W} \mathbb{Q}\sigma(w),
\end{align*}
we have $\dim_{\mathbb{Q}}(H^*(G/B)/J_{H})=\vert X_{H}(e_{\theta})^T\vert$.
Also, the Schubert cell decomposition of $X_H(e_{\theta})$ gives $\dim_{\mathbb{Q}}(H^*(X_H(e_{\theta})))=\vert X_{H}(e_{\theta})^T\vert.$ 
Hence, \eqref{Dimensions} is satisfied and the map $H^*(G/B)/J_{H} \rightarrow H^*(X_H(e_{\theta}))$ is an isomorphism.
\qed

\vspace{10pt}

For example, suppose that $G=\SL_3(\mathbb{C})$ and that all notation is as presented in \ref{Section: Examples in Type A}. We will use Theorem \ref{Ordinary Kernel} to obtain a presentation of $H^*(X_{\mathfrak{b}}(e_{\theta}))$, the cohomology ring of our Springer fiber. To this end, let $Flags(\mathbb{C}^3) \times \mathbb{C}^3$ be the trivial vector bundle over $Flags(\mathbb{C}^3)$, and set
\begin{align*}
E_i := \{ (V_{\bullet},v)\in Flags(\mathbb{C}^3)\times\mathbb{C}^3 \mid v\in V_i \}
\end{align*}
for $i\in\{1,2,3\}$. Note that $E_i$ is a complex vector bundle over $Flags(\mathbb{C}^3)$. Each quotient $L_i := E_i/E_{i-1}$ is a complex line bundle, allowing us to consider its first Chern class $$c_1(L_i)\in H^*(Flags(\mathbb{C}^3)).$$ Now, recall that the algebra morphism  
$$\mathbb{Q}[x_1,x_2,x_3]\rightarrow H^*(Flags(\mathbb{C}^3)),\quad x_i\mapsto c_1(L_i),\text{ } i\in\{1,2,3\}$$
is surjective. Recall also that its kernel is the ideal generated by $e_1(x)$, $e_2(x)$, and $e_3(x)$, where $e_i(x)$ denotes the $i$-th elementary symmetric polynomial in the variables $x_1,x_2,x_3$. In particular, we have an algebra isomorphism
\begin{equation}\label{Flag Variety Cohomology Isomorphism}
H^*(Flags(\mathbb{C}^3))\xrightarrow{\cong}\mathbb{Q}[x_1,x_2,x_3]/(e_1(x),e_2(x),e_3(x)).
\end{equation}
The ideal $J_{\mathfrak{b}}\subseteq H^*(Flags(\mathbb{C}^3))$ is seen to be generated by the opposite Schubert classes $\sigma(2\ 3\ 1),\sigma(3\ 1\ 2),\sigma(3\ 2\ 1)\in H^*(Flags(\mathbb{C}^3))$. Their images under the isomorphism \eqref{Flag Variety Cohomology Isomorphism} are 
\begin{align*}
\sigma(2\ 3\ 1) = x_1x_2, \quad 
\sigma(3\ 1\ 2) = x_1x_1, \quad
\sigma(3\ 2\ 1) = x_1x_1x_2, 
\end{align*}
where (by an abuse of notation) $x_i$ is also used to denote its image in the quotient algebra $\mathbb{Q}[x_1,x_2,x_3]/(e_1(x), e_2(x), e_3(x))$. Applying Theorem \ref{Ordinary Kernel}, we obtain
$$H^*(X_{\mathfrak{b}}(e_{\theta}))\cong H^*(Flags(\mathbb{C}^3))/J_{\mathfrak{b}}\cong\frac{\mathbb{Q}[x_1,x_2,x_3]/(e_1(x), e_2(x), e_3(x))}{\mathbb{Q}x_1x_2\oplus\mathbb{Q}x_1x_1\oplus\mathbb{Q}x_1x_1x_2}.$$
A straightforward manipulation of the rightmost ring then yields
$$H^*(X_{\mathfrak{b}}(e_{\theta}))\cong \mathbb{Q}[x_1,x_2,x_3]/(e_1(x), e_2(x), e_3(x), x_1x_2, x_1x_3, x_2x_3),$$ which is exactly Tanisaki's presentation of $H^*(X_{\mathfrak{b}}(e_{\theta}))$ (see \cite{Tanisaki}).

\subsection{Equivariant Cohomology}\label{Section: Equivariant Cohomology}
As one might expect, we have the following equivariant counterpart of Proposition \ref{Surjectivity on Ordinary Cohomology}.
\begin{proposition}\label{equiv surjectivity}
The restriction map $i_T^*:H_T^*(G/B)\rightarrow H_T^*(X_H(e_{\theta}))$ is surjective.
\end{proposition}
\proof
It will suffice to prove that the restriction of $i_T^*$ to degree--$k$ cohomology, $H_T^{k}(G/B)\rightarrow H_T^{k}(X_H(e_{\theta}))$, is surjective for all $k\geq 0$. To accomplish this, we will use induction on $k$.

For the base case, note that $G/B$ and $X_H(e_{\theta})$ are connected (see Theorem 4.4 of \cite{PrecupConnectedness} for the connectedness of $X_H(e_{\theta})$). It follows that $i_T^*$ is surjective on degree--$0$ cohomology. Now, assume that $i_T^*$ is surjective on degree--$j$ cohomology for all $j\leq k$ and let $\alpha\in H_T^{k+1}(X_H(e_{\theta}))$ be given. Since $G/B$ and $X_H(e_{\theta})$ are equivariantly formal, the forgetful maps $\phi:H_T^*(G/B)\rightarrow H^*(G/B)$ and $\phi':H_T^*(X_H(e_{\theta}))\rightarrow H^*(X_H(e_{\theta}))$ fit into the commutative diagram
$$\begin{CD}
0 @>{ }>> H_T^{>0}(\text{pt})H_T^*(G/B) @>{}>> H_T^*(G/B) @>{\phi}>> H^*(G/B)@>{}>>0\\
@. @V{ }VV @V{i_T^*}VV @V{i^*}VV\\
0 @>{ }>> H_T^{>0}(\text{pt})H_T^*(X_H(e_{\theta})) @>{}>>H_{T}^*(X_H(e_{\theta}))@>{\phi'}>> H^*(X_H(e_{\theta}))@>{}>>0
\end{CD}$$
of exact sequences.
By a straightforward diagram chase, there exists $\beta\in H_T^{k+1}(G/B)$ such that $\phi'(\alpha-i_T^*(\beta))=0$. It follows that $\alpha-i_T^*(\beta)\in H_T^{>0}(\text{pt})H_T^*(X_H(e_{\theta}))$, or equivalently
\begin{equation}\label{Equation: Linear Combination}
\alpha-i_T^*(\beta)=\sum_{j=1}^nc_j\gamma_j
\end{equation} for some homogeneous $c_1,\ldots,c_n\in H_T^{>0}(\text{pt})$ and some homogeneous $\gamma_1,\ldots,\gamma_n\in H_T^{*}(X_H(e_{\theta}))$. Since $\alpha-i_T^*(\beta)\in H^{k+1}_T(X_H(e_{\theta}))$, it follows that $\gamma_1,\ldots,\gamma_n$ are of degree $\leq k$. Using our surjectivity assumption, we may find homogeneous elements $\omega_1,\ldots,\omega_n\in H_T^{*}(G/B)$ for which $\gamma_1=i_T^*(\omega_1),\ldots,\gamma_n=i_T^*(\omega_n)$. Hence, \eqref{Equation: Linear Combination} becomes

$$\alpha-i_T^*(\beta) =\sum_{j=1}^nc_ji_T^*(\omega_j)=i_T^*\left(\sum_{j=1}^nc_j\omega_j\right).$$ In other words, 
$$\alpha=i_T^*\left(\beta+\sum_{j=1}^nc_j\omega_j\right),$$
completing the proof.
\qed

Proceeding in analogy with \ref{Section: Ordinary Cohomology}, we now compute the kernel of $i_T^*$. However, we will need the following well-known description of the image of $\sigma_T(w)$ under the restriction map 
\begin{equation}\label{Definition of Restriction to a Point} i_w^*:H_T^*(G/B)\rightarrow H_T^*(\{x_w\})=H_T^*(\text{pt}).
\end{equation}

\begin{lemma}\label{Restriction of an Equivariant Schubert Class}
If $w\in W$, then \begin{equation}
\label{Restriction to Fixed Point}
i_w^*(\sigma_T(w))=\prod_{\alpha\in\Delta_{+}\cap w\Delta_{-}}\alpha.\end{equation}
\end{lemma}

\begin{proof}
Since $x_w$ is a smooth point of $X_{-}(w)$, $i_w^*(\sigma_T(w))$ is precisely the $T$-equivariant Euler class of the $T$-representation \begin{equation}\label{T-Representation}
T_{x_w}(G/B)/T_{x_w}(X_{-}(w)).\end{equation} It is therefore equal to the product of the weights occurring in \eqref{T-Representation}, which we now determine. To this end, as $wBw^{-1}$ is the $G$-stabilizer of $x_w$ and has Lie algebra $w\mathfrak{b}$, we have isomorphisms \begin{equation}\label{Tangent Space to Flag Variety}T_{x_w}(G/B)\cong\mathfrak{g}/w\mathfrak{b}\cong \bigoplus_{\alpha\in w\Delta_{-}}\mathfrak{g}_{\alpha}\end{equation} of $T$-representations. Also, the $B_{-}$-stabilizer of $x_w$ is $B_{-}\cap wBw^{-1}$ and has Lie algebra $\mathfrak{b}_{-}\cap w\mathfrak{b}$. We therefore have \begin{equation}\label{Tangent Space to Schubert Variety}
T_{x_w}(X_{-}(w))=T_{x_w}(B_{-}wB/B)\cong\mathfrak{b}_{-}/(\mathfrak{b}_{-}\cap w\mathfrak{b})\cong\bigoplus_{\alpha\in\Delta_{-}\cap w\Delta_{-}}\mathfrak{g}_{\alpha}.\end{equation} Combining \eqref{Tangent Space to Flag Variety} and \eqref{Tangent Space to Schubert Variety}, one finds that $$T_{x_w}(G/B)/T_{x_w}(X_{-}(w))\cong\bigoplus_{\alpha\in\Delta_{+}\cap w\Delta_{-}}\mathfrak{g}_{\alpha}$$ as $T$-representations. This completes the proof.  
\end{proof}  
   
Given $w\in W$, note that $X_{-}(w)$ determines an equivariant opposite Schubert class $$\sigma_T(w)\in H_T^{2\ell(w)}(G/B).$$ These classes are seen to form an $H_T^*(\text{pt})$-module basis of $H_T^*(G/B)$. With this in mind, the following corollary introduces an important $H_T^*(\text{pt})$-submodule of $H_T^*(G/B)$.
\begin{corollary}
The submodule \begin{equation}\label{Equivariant Ideal}
J_H^T:=\bigoplus_{x_w\not\in X_H(e_{\theta})^T}H_T^*(\emph{pt})\sigma_T(w)
\end{equation} is an ideal of $H_T^*(G/B)$.
\end{corollary}

\begin{proof}

The argument is similar to that used in the proof of Proposition \ref{corollary for ideal}, provided one uses the well-known fact that
\begin{equation}\label{Linear Combination}
\sigma_T(u)\sigma_T(w)=\sum_{v\geq u,w}c_{uw}^v\sigma_T(v)
\end{equation}
for $c_{uw}^v\in H_T^*(\text{pt})$. For the reader's convenience, we briefly recount a proof of this fact. To this end, let $v\in W$ be a minimal element with the property that \begin{equation}\label{Minimal Element}i_{v}^*(\sigma_T(u)\sigma_T(w))\neq 0.\end{equation} Note that $vs_{\alpha}<v$ for all $\alpha\in\Delta_{+}\cap v^{-1}\Delta_{-}$, so that $$i_{vs_{\alpha}}^*(\sigma_T(u)\sigma_T(w))=0,\text{ }\text{ }\alpha\in\Delta_{+}\cap v^{-1}\Delta_{-}.$$ The GKM conditions \eqref{Image of Localization Map for G/B} defining the image of $H_T^*(G/B)\rightarrow H_T^*((G/B)^T)$ then give $$(v\alpha)\text{ }\vert \text{ }i_{v}^*(\sigma_T(u)\sigma_T(w)),\text{ }\text{ }\alpha\in\Delta_{+}\cap v^{-1}\Delta_{-}.$$ Hence, the product 
\begin{equation}\label{Root Product}\prod_{\alpha\in\Delta_{+}\cap v^{-1}\Delta_{-}}v\alpha\end{equation}
also divides $i_{v}^*(\sigma_T(u)\sigma_T(w))$. Using Lemma \ref{Restriction of an Equivariant Schubert Class}, one finds that \eqref{Root Product} coincides with\\ $(-1)^{l(v)}i_v^*(\sigma_T(v))$. In particular, $i_v^*(\sigma_T(v))$ divides $i_v^*(\sigma_T(u)\sigma_T(w))$, meaning that 
\begin{equation}\label{Support Reduction}i_v^*(\sigma_T(u)\sigma_T(w)-c_{uw}^v\sigma_T(v))=0\end{equation} for some $c_{uw}^v\in H_T^*(\text{pt})$. 

Continuing the support-reducing process by induction, one eventually obtains a class with no support in the GKM graph. In other words, there exist coefficients $c_{uw}^v\in H_T^*(\text{pt})$ for all $v\geq u,w$ such that $$\sigma_T(u)\sigma_T(w)-\sum_{v\geq u,w}c_{uw}^v\sigma_T(v)$$ has zero image under the localization map $H_T^*(G/B)\rightarrow H_T^*((G/B)^T)$. Since the localization map is injective, we conclude that \begin{equation}\label{Linear Combination}
\sigma_T(u)\sigma_T(w)=\sum_{v\geq u,w}c_{uw}^v\sigma_T(v).
\end{equation}
\end{proof}

\begin{theorem}\label{Equivariant Kernel}
The map $i_T^*:H_T^*(G/B)\rightarrow H_T^*(X_H(e_{\theta}))$ induces a graded $H_T^*(\emph{pt})$-algebra isomorphism $$H_T^*(G/B)/J_H^T\rightarrow H_T^*(X_H(e_{\theta})).$$
\end{theorem}

\begin{proof}
Having established \eqref{Empty Intersection} in the proof of Theorem \ref{Ordinary Kernel}, we see that $i_T^*(\sigma_T(w))=0$ for $x_w\not\in X_H(e_{\theta})^T$.  
Therefore, $i_T^*$ induces a surjective map
\begin{align}\label{equiv restriction map}
H_T^*(G/B)/J_H^T\rightarrow H_T^*(X_H(e_{\theta})).
\end{align}

Now, from the definition of $J_H^T$, it is clear that 
\begin{align*}
H_T^*(G/B)/J_H^T\cong\bigoplus_{x_w\in X_H(e_{\theta})^T}H_T^*(\text{pt})\sigma_T(w)
\end{align*}
as $H_T^*(\text{pt})$-modules. In particular, $H_T^*(G/B)/J_H^T$ is free of rank $\vert X_H(e_{\theta})^T\vert$.  
However, as $X_H(e_{\theta})$ is $T$-equivariantly formal, $H_T^*(X_H(e_{\theta}))$ is also free of rank $\vert X_H(e_{\theta})^T\vert$. It follows that \eqref{equiv restriction map} is actually an isomorphism.
\end{proof}

\bibliographystyle{acm} 
\bibliography{Nilpotent}
\end{document}